\documentclass[reqno,10pt]{amsart}
\usepackage{ esint }

\usepackage{overpic}
\usepackage{amssymb}
\usepackage{amsmath}
\usepackage{amsthm}
\usepackage{mathrsfs}
 \usepackage{tikz}
\usetikzlibrary{shapes,snakes,spy,calc,shapes.geometric}

\usepackage{pgf}
\usepackage{color}
\usepackage{graphicx}   
\usepackage{hyperref}

\usepackage{amsmath}
  \usepackage{paralist}
\usepackage{graphicx}  
\usepackage{psfrag}
\usepackage{comment}
\usepackage{esvect}

\newtheorem{theorem}{Theorem}[section]
\newtheorem{corollary}[theorem]{Corollary}

\newtheorem{lemma}[theorem]{Lemma}
\newtheorem{proposition}[theorem]{Proposition}

\theoremstyle{definition}
\newtheorem{definition}[theorem]{Definition}
\newtheorem{remark}[theorem]{Remark}

\newtheorem{example}[theorem]{Example}

\numberwithin{equation}{section}

\newcommand{\R}{\mathbb{R}}
\newcommand{\N}{\mathbb{N}}

\newcommand{\eps}{\varepsilon}

\newcommand{\ove}{\overline}

\newcommand{\mres}{\mathbin{\vrule height 1.6ex depth 0pt width
0.13ex\vrule height 0.13ex depth 0pt width 1.3ex}}

\newcommand{\BBB}{\color{black}}

\newcommand{\EEE}{\color{black}}

\usepackage{latexsym,amsfonts}
\usepackage{mathrsfs}
\usepackage{centernot}
\usepackage{paralist}

\usepackage{eso-pic}  
\usepackage{pifont}
\usepackage{fixmath}

\numberwithin{equation}{section}

\begin{document}

\title[{Lower semicontinuity  for functionals in $PR$  and $GSBD$}]{Lower semicontinuity  for functionals defined  on piecewise rigid functions and on $GSBD$}



\author{Manuel Friedrich}
\address[Manuel Friedrich]{Applied Mathematics,  
Universit\"{a}t M\"{u}nster, Einsteinstr. 62, D-48149 M\"{u}nster, Germany}
\email{manuel.friedrich@uni-muenster.de}
\urladdr{https://www.uni-muenster.de/AMM/Friedrich/index.shtml}

\author{Matteo Perugini}
\address[Matteo Perugini]{Applied Mathematics,  
Universit\"{a}t M\"{u}nster, Einsteinstr. 62, D-48149 M\"{u}nster, Germany}
\email{matteo.perugini@uni-muenster.de}

\author{Francesco Solombrino}
\address[Francesco Solombrino]{Dip. Mat. Appl. ``Renato Caccioppoli'', Univ. Napoli ``Federico II'', Via Cintia, Monte S. Angelo
80126 Napoli, Italy}
\email{francesco.solombrino@unina.it}
\urladdr{http://www.docenti.unina.it/francesco.solombrino}

\begin{abstract}
In this work, we provide a characterization result for lower semicontinuity of surface energies defined on  piecewise rigid functions, i.e., functions which are piecewise affine on a Caccioppoli
partition where the derivative in each component is a skew symmetric matrix. This characterization is achieved by means of an integral condition, called  \emph{$BD$-ellipticity}, which is in the spirit of $BV$-ellipticity defined by Ambrosio and Braides \cite{AmbrosioBraides2}. By specific examples we show that this novel concept is in fact stronger compared to its $BV$ analog.  We further provide  a sufficient condition implying $BD$-ellipticity which we call \emph{symmetric joint convexity}. This notion can be checked explicitly for certain classes of surface energies which are relevant for applications, e.g., for variational fracture models.  
Finally, we give a direct proof that surface energies with symmetric jointly convex integrands are lower semicontinuous also on the larger space of $GSBD^p$ functions.  
\end{abstract}
\maketitle

\section{Introduction}\label{sec: introduction}

The minimization of surface energies for configurations which represent partitions of the domain into regions of finite perimeter appears in 
many  problems in  materials science, physics, computer science, and other fields (see, for instance, \cite[Introduction]{FM} and the references therein).  
In the framework of the calculus of variations, these energies are often given in the form of integral functionals defined on \emph{Caccioppoli partitions} or \emph{piecewise constant functions} on such partitions, see \cite[Section 4.4]{Ambrosio-Fusco-Pallara:2000}  for their definition.   After the seminal work by {\sc Almgren} \cite{Almgren}, {\sc Ambrosio and Braides} \cite{AmbrosioBraides, AmbrosioBraides2} developed a thorough analysis concerning integral representation, compactness, $\Gamma$-convergence, and relaxation for this class of functionals. They also formulated a general theory of lower semicontinuity in this setting,  which we will discuss in detail below. This approach was further developed by subsequent contributions over the last years, see, e.g.,   \cite[Section 5.3]{Ambrosio-Fusco-Pallara:2000} or \cite{Caraballo1, Caraballo2, Caraballo3}. Let us also mention some recent advances dealing with density and continuity results \cite{BCG-part, matthias}, witnessing that the study of this class of functionals is of ongoing interest.

 \textbf{Background:}  We now briefly discuss the framework for lower semicontinuity devised in \cite{AmbrosioBraides2} since it will be relevant for the purpose of our paper. There, integral functionals of the form 
\begin{equation}\label{eq: BVmodel}
 u \mapsto  \int_{J_u \cap \Omega} f(u^+, u^-, \nu_u)\,\mathrm{d}\mathcal{H}^{d-1}
\end{equation}
 are   considered. Above,   $u=\sum_{ k \in \N  } b_k \chi_{P_k}$ is a piecewise constant function  where the  sets $P_k$ partition a $d$-dimensional reference configuration $\Omega$ into subsets of finite perimeter. Thus, the jump set $J_u$ of $u$, locally oriented by a normal unit vector $\nu_u$, consists of the interfaces between two different $P_k$'s where $u$ jumps from the value $u^+$ to $u^-$. The  constants  $b_k$ are taken from a prescribed finite subset $T$ of $\mathbb{R}^m$, so that  without restriction one can  assume the integrand $f\colon \mathbb{R}^m \times  \mathbb{R}^m \times \mathbb{S}^{d-1}  \to[0,+\infty)  $ to be continuous and bounded.    (Here, $\mathbb{S}^{d-1}$ denotes the unit sphere in $\R^d$.) In \cite{AmbrosioBraides2},  it is shown  with a localization technique  that lower semicontinuity of energies of the kind \eqref{eq: BVmodel}  can be equivalently reformulated in terms of an integral condition, named $BV$-ellipticity.  The latter   plays a similar role as Morrey quasi-convexity \cite{Morrey} for integral functionals on Sobolev spaces. It  requires  that, for all $(i,j, \nu) \in T\times T \times \mathbb{S}^{d-1}$  with $i \neq j$,  we have
\[
\int_{J_v \cap  Q^\nu_1  } f(v^+, v^-, \nu_v)\,\mathrm{d}\mathcal{H}^{d-1}\ge f(i, j, \nu)
\]
for all piecewise constant functions $v$ (with values in $T$) which take the values $i$ and $j$, respectively, on the upper and the lower part of the boundary of  $Q^\nu_1$ which is  the unit cube  in $\R^d$  oriented by $\nu$.    This condition, however, is not easy to handle because it is given by an integral
inequality. To overcome this difficulty, a  sufficient condition  for semicontinuity  has been  introduced which can be easily verified in many practical cases: it is  called {\it regular biconvexity} or  {\it joint convexity} (this latter expression is used in the reference book \cite{Ambrosio-Fusco-Pallara:2000}) and amounts to require that
\begin{align}\label{eq: joinijoini}
f(i,j, \nu)=\sup_{h \in \mathbb{N}}\ \langle V_h(i)-V_h(j), \nu \rangle,
\end{align}
 where  $(V_h)_h \subset  C_0(\mathbb{R}^m, \mathbb{R}^d)$ is a countable collection of continuous functions vanishing at infinity.  

Understanding the properties of functionals on Caccioppoli partitions has also proved to be a fundamental step in the analysis of free-discontinuity problems \cite{Ambrosio-Fusco-Pallara:2000, DeGiorgi-Ambrosio:1988} defined on  \emph{(generalized) special functions of bounded variation} ($(G)SBV$) (see \cite[Section 4]{Ambrosio-Fusco-Pallara:2000}). The study of lower semicontinuity conditions for surface energies of the form \eqref{eq: BVmodel}, but considered in the larger space $GSBV$, is indeed one of the relevant issues that can be reduced to corresponding problems on partitions, see \cite{Amb, Ambrosio:90-2}. Since piecewise constant functions are a subset of $GSBV$, it is clear that $BV$-ellipticity still provides a necessary condition for lower semicontinuity in  a  suitable weak topology, essentially the one where {\sc Ambrosio's} compactness theorem \cite[Theorems 4.7-4.8]{Ambrosio-Fusco-Pallara:2000} holds. Remarkably,  using an approximation argument of $SBV$ functions with piecewise constant ones  (which essentially relies on the $BV$ coarea formula),  \cite[Theorem 3.3]{Amb}   shows that $BV$-ellipticity actually provides also a sufficient condition for lower semicontinuity along sequences which are uniformly bounded both in $L^\infty$   and in the weak topology of $SBV$.    If $L^\infty$ bounds are not available  and one still wants to allow for possibly unbounded integrands, lower semicontinuity results in $GSBV$ can be provided only under additional structural assumptions,  see \cite[Theorem 3.7]{Amb},  since the traces $u^+$ and $u^-$ are  not necessarily integrable on $J_u$  in this case. In the proof, the possibility of using Lipschitz truncations to approximate $GSBV$ with bounded $SBV$ functions plays a relevant role,  a tool which is not available in our setting described below.

 \textbf{Setting of the paper:}  In the present paper, we are interested in analogous problems for functionals defined on \emph{piecewise rigid functions},  denoted by $PR(\Omega)$,  i.e., functions which are piecewise affine on a Caccioppoli partition  where the  derivative in each component is constant and  lies  in the set of skew symmetric matrices $ \R^{d\times d}_{\rm skew}$. Functions in this space are vector-valued and take the form
$$
{u(x) = \sum\nolimits_{k\in \N}  (Q_k\,  x + b_k) \chi_{P_k}(x),}
$$     
 where $(P_k)_{k\in\N}$ is a Caccioppoli partition of $\Omega$,  $ Q_k  \in  \R^{d\times d}_{\rm skew}$, and $b_k \in \R^d$  for all $k \in \N$.   Due to a remarkable \emph{piecewise rigidity} result   in \cite{Chambolle-Giacomini-Ponsiglione:2007}, the set of these  functions coincides 
with the (seemingly larger) set of functions $u \in GSBD$ with approximate symmetrized gradient $e(u)=0$ almost everywhere. Here, $(G)SBD$ is the space of \emph{(generalized) special functions of bounded deformation}, introduced in \cite{ACD, DM}. 
Actually, our primary motivation comes  exactly from the study of  free-discontinuity problems defined on the space $GSBD^p$, see  \cite{DM},  which has obtained steadily  increasing attention over the last years, cf., e.g., \cite{Chambolle-Conti-Francfort:2014, Iu3, Crismale2, Crismale,  Conti-Iurlano:15, Conti-Focardi-Iurlano:15, CoFoIu, Vito, Crismale-Friedrich, Friedrich:15-2, Friedrich:15-3, Friedrich:15-4, FM, FriedrichSolombrino}.  (Here, the exponent $p$ refers to summability of the  approximate  symmetrized gradient.)  In these problems, only a control on the symmetrized gradient of the  admissible configurations  is available. Hence, a larger space than piecewise constant functions must be taken into account in order to provide lower semicontinuity conditions for surface integrands. It is quite natural to expect (and indeed our results in Sections \ref{sec: counterexamples} and \ref{sec: GSBD} will justify this point of view) that the understanding of energies defined on piecewise rigid functions  is a significant ingredient  of such a research program. 

The results  of this paper (see description below)  also complement the ones we obtained in a first paper on this topic \cite{FM}, where integral representation and $\Gamma$-convergence for functionals defined on piecewise rigid functions have been investigated. The proof strategy there  was  based on the global method for relaxation developed in \cite{BFLM, BDM},  but  some highly nontrivial issues had to be faced. In particular, a key ingredient for the results in \cite{FM} (and actually also for the ones in the present paper)  is a construction for  \emph{joining} two functions $u,v \in PR(\Omega)$, which is usually called the \emph{fundamental estimate}. In    the space $PR(\Omega)$,  this cannot be achieved by means of a cut-off construction of the form $w := u \varphi + (1 - \varphi)  v  $ for some smooth $\varphi$ with $0 \le \varphi \le 1$, since in general $w$ is not in $PR(\Omega)$. In the case of piecewise constant functions, this issue was solved in \cite{AmbrosioBraides} by using the coarea formula in $BV$, see \cite[Lemma 4.4]{AmbrosioBraides}, a tool which is not available in spaces of functions with bounded deformation.  However,  the issue can be successfully overcome: a statement of the fundamental estimate  in  $PR(\Omega)$  is given in Lemma  \ref{lemma: fundamental estimate2-new}, while we refer the interested reader to \cite[Introduction, Lemmas 4.1 and 4.4]{FM} for an overview of the proof strategy, and the detailed proof, respectively.

 \textbf{Results of the paper:}  We now come to the description of our results. We essentially follow the program of \cite{AmbrosioBraides2}. On the space $PR(\Omega)$,  we consider functionals of the form 
\begin{equation}\label{eq: BDmodel}
 u \mapsto  \int_{J_u \cap \Omega} f(u^+, u^-, \nu_u)\,\mathrm{d}\mathcal{H}^{d-1}
\end{equation}
 for densities $f\colon \R^d\times \R^d \times \mathbb{S}^{d-1} \to [0,+\infty)$.  We introduce an integral notion which we call \emph{BD-ellipticity}: it requires that for all $(i,j, \nu) \in \mathbb{R}^d\times \mathbb{R}^d \times \mathbb{S}^{d-1}$ with $i \neq j$ we have
\[
\int_{J_v \cap  Q^\nu_1 } f(v^+, v^-, \nu_v)\,\mathrm{d}\mathcal{H}^{d-1}\ge f(i, j, \nu)
\]
for all $v\in  PR(\Omega)$ which take the values $i$ and $j$, respectively, on the upper and the lower part of the boundary of the unit cube oriented by $\nu$,  again denoted by $Q^\nu_1$.  The fundamental estimate in Lemma \ref{lemma: fundamental estimate2-new} and a localization procedure allows us to show that this condition is equivalent to lower semicontinuity of the functionals \eqref{eq: BDmodel} along sequences which converge in measure and whose jump sets have uniformly bounded  $\mathcal{H}^{d-1}$-measure,  provided  that  $f$ is uniformly continuous and bounded (Theorem \ref{thm: LSC iff BD-ell + bounded}). If this latter requirement is dropped, lower semicontinuity is still guaranteed on sequences which are bounded in $L^\infty$,  see Corollary \ref{cor: unbounded}. However, this  uniform bound  is quite a restrictive assumption in the variational modeling of fracture where the interest in functionals of the kind \eqref{eq: BDmodel} originates. Hence, it is relevant to know that we may get rid of $L^\infty$-bounds also in some cases where $f$ is unbounded, as we point out in Corollary \ref{cor: f=sup fh is BD and LSC}. There, we prove that, if $f$ is the supremum of $BD$-elliptic, uniformly continuous, and bounded integrands, it is itself $BD$-elliptic and the associated functional is lower semicontinuous. 

 In case $f$ is not $BD$-elliptic, we also study the relaxation of functionals of the form  \eqref{eq: BDmodel}, see Theorem \ref{thm: don't worry be happy} for details.  After providing an abstract framework for lower semicontinuity, the rest of the paper is devoted to investigate more closely   related notions. On the one hand, our objective is to provide sufficient conditions and  to find relevant explicit examples of functionals fulfilling our assumptions. On the other hand, we  compare the notion of $BD$-ellipticity with its $BV$ analog: as one may expect, we show that  it is   actually more restrictive. 

 To provide a sufficient notion for lower semicontinuity, we introduce  a suitable subclass of jointly convex integrands, which we call {\it symmetric jointly convex} functions: they are still of the form
\begin{equation}\label{eq: sjc}
f(i,j, \nu)=\sup_{h \in \mathbb{N}} \ \langle V_h(i)-V_h(j), \nu \rangle\,,
\end{equation}
 cf.\ \eqref{eq: joinijoini},  but, along  with boundedness and uniform continuity, we additionally require the vector fields $V_h\colon \R^d\to \R^d$  to be \emph{conservative}. The role of this condition, which for smooth fields $V_h$ implies that the differentials $DV_h$ are symmetric matrices, is apparent in the proof of Theorem \ref{thm: LSC symjointconvex} as it implies that the distributional divergence of the composite functions $V_h(u)$, with $u \in PR(\Omega)$, is concentrated on the jump set $J_u$. This allows to reproduce the arguments in \cite{AmbrosioBraides2},   see also \cite[Theorem 5.20]{Ambrosio-Fusco-Pallara:2000}, and to prove that symmetric joint convexity is a sufficient condition  both  for  $BD$-ellipticity and for the lower semicontinuity of the associated  integral functionals \eqref{eq: BDmodel}. 

With symmetric joint convexity at our disposal, we can give explicit examples of functionals complying with our setting, see Section \ref{sec: examples}. For instance, in Subsection \ref{sec: joint2} we show that the functionals frequently  used to describe \emph{cohesive fracture} energies,  namely 
\[
 u \mapsto  \int_{J_u} |u^+-u^-|\,\mathrm{d}\mathcal{H}^{d-1}\,,\quad \quad  u \mapsto  \int_{J_u}   \min \lbrace |u^+-u^-|, M \rbrace  \,\mathrm{d}\mathcal{H}^{d-1},
\]
 are lower semicontinuous. The key observation to this purpose is that for all $(i, j, \nu)$ it holds
\[
|i-j|\,|\nu|=\sup \,  \langle Bi- Bj, \nu \rangle,
\]
where the supremum is taken over all symmetric matrices $B$ having {\it operator norm} at most $1$,  see Lemma \ref{lem: Bu=v}.   Note that  the vector field $x \mapsto Bx$ is conservative by symmetry of the matrix $B$.  Then,  by decomposing the action of $B$ into one-dimensional, orthogonal eigenspaces and truncating the resulting functions of one variable, we can approximate  $x\mapsto Bx$  from below with bounded,  uniformly continuous, and  conservative vector fields  $V_h$. This shows that the integrands are symmetric jointly convex,   which yields the desired   lower semicontinuity.   The results in Subsection \ref{sec: joint2} also apply  to  more general surface integrands of the form $g(|i-j|)|\nu|$, for  increasing subadditive functions $g$. (Actually, an  additional restriction has to be imposed, cf.\ the statement of Theorem \ref{example: g subadditive}.) Notice that these integrands are {\it isotropic}, in  contrast  to similar ones considered in \cite{AmbrosioBraides2} which (as we will discuss later) may instead fail to be $BD$-elliptic.

\BBB 
Examples of lower semicontinuous energies on $BD$ are very rare to date, and results appear to be limited to \cite{DalMOT, GZ1, GZ2, GZ3}. In \cite{GZ1}, 
functionals of the form 
$$ u \mapsto  \int_{J_u}    \varphi\big( |\langle u^+-u^-, \nu_u \rangle| \big) \,\mathrm{d}\mathcal{H}^{d-1} $$
for convex, subadditive, and increasing densities $\varphi \colon [0,\infty) \to [0,\infty)$  are considered (with an additional noninterpenetration constraint), and \cite{GZ2}  deals with surface energies of the form 
$$ u \mapsto  \int_{J_u} \sup_{\xi \in \mathbb{S}^{d-1}}   |\langle \nu_u, \xi \rangle  |  \psi\big( |\langle u^+-u^-, \xi \rangle| \big) \,\mathrm{d}\mathcal{H}^{d-1}, $$
where $\psi \colon [0,\infty) \to [0,\infty)$ is a lower semicontinuous, nondecreasing, and  subadditive function. In \cite{DalMOT}, a class of functionals is investigated which particularly includes 
\begin{align}\label{eq: first}
 u \mapsto  \int_{J_u} |(u^+-u^-)\odot \nu_u|\,\mathrm{d}\mathcal{H}^{d-1}\,,\quad \quad  u \mapsto  \int_{J_u}G_M\big((u^+-u^-)\odot \nu_u\big)\,\mathrm{d}\mathcal{H}^{d-1},
\end{align}
where  $|\cdot|$  is the Frobenius norm, $\odot$ denotes the symmetrized tensor product,  and the second integrand is a suitable truncation of the first one, explicitly calculated in \cite[Section 6]{DalMOT}. In Subsections \ref{sec: joint2} and  \ref{sec: orlando},  we show that all densities above are symmetric jointly convex functions, and  therefore the above classes can  be treated by our approach.  \EEE

After further examples in Subsections \ref{sec: biconvex}--\ref{sec:O},  in Subsection \ref{sec: counterexamples} we instead address the comparison between $BD$- and $BV$-ellipticity. In particular, we show with a counterexample  (Example \ref{counterexample})  that  anisotropic integrands of the form $|i-j|\psi(\nu)$, where $\psi$ denotes a  norm different from the Euclidean one, are in general not $BD$-elliptic\footnote{The counterexample can be  extended straightforwardly to the case $\theta(|i-j|)\psi(\nu)$ for a suitable bounded $\theta$ if one does not want to cope with unbounded integrands.}. As these functions  are known to be $BV$-elliptic, the associated functionals are lower semicontinuous in the $SBV$-weak topology considered  in \cite{Amb},  but not in the analogous topology in the space of special functions of bounded deformation. A similar counterexample can also be provided for the case of integrands which are anisotropic in the jump direction, see Example \ref{counterexample2}.

Whereas $BD$-ellipticity provides a complete theoretical framework for lower semicontinuity of surface energies in $PR(\Omega)$ and can be used for providing counterexamples in the larger space $GSBD^p$, a  complete characterization of  lower semicontinuity in $GSBD^p$ is still missing. Indeed,  unlike the $BV$ case, we cannot reconduct in general  the problem to the analogous one for piecewise rigid functions, essentially due to the lack of a "coarea-like" formula in our setting. However, our results can be succesfully exploited to tackle the well-posedness of miminum problems for energies in $GSBD^p$, provided we assume the surface integrands to be symmetric jointly convex. In fact, we can give a direct proof that functionals of the form \eqref{eq: BDmodel}, with $f$ symmetric jointly convex, are lower semicontinuous along sequences of $GSBD^p$ functions which converge in measure, whose jump sets have uniformly bounded $\mathcal{H}^{d-1}$-measure, and whose symmetrized gradient have equibounded $L^p$ norm,  see Theorem \ref{thm: GSBD LSC}. The proof makes use of an integration-by-parts formula, which in its turn relies on the fact that the vector fields in \eqref{eq: sjc} are conservative.  The latter ensures that  the Lebesgue part of the distributional divergence of the composite functions $V_h(u)$, with $u \in GSBD(\Omega)$, only depends on $V_h$, $u$, and the symmetrized gradient $e(u)$, see Lemma \ref{lemma: BD-int-by-parts}.   Eventually, this  allows us to successfully adapt the localization procedure  of \cite[Theorem 3.6]{Amb}  to our setting. \BBB Let us mention that this lower semicontinuity result is an important ingredient to characterize relaxations of variational problems defined on
 $GSBD^p$ \cite{FPM}. \EEE

As a final remark, we point out that, if we combine the above-mentioned results of Sections \ref{sec: examples} and \ref{sec: GSBD} with  compactness in  $GSBD$  \cite[Theorem 11.3]{DM}, we obtain the well-posedness  of some variational problems of relevant applied interest,  such as 
\[
\begin{split}
&  u \mapsto \int_\Omega |e(u)|^p\,\mathrm{d}x+\int_{J_u}(1+ |u^+-u^-|)\,\mathrm{d}\mathcal{H}^{d-1}+\int_\Omega  \Psi(|u|)  \,\mathrm{d}x\,,\\
& u \mapsto\int_\Omega |e(u)|^p\,\mathrm{d}x+\int_{J_u} \big(1+ \min\lbrace |u^+-u^-|, M\rbrace \big)  \,\mathrm{d}\mathcal{H}^{d-1}+\int_\Omega \Psi(|u|)\,\mathrm{d}x\\
& u \mapsto\int_\Omega |e(u)|^p\,\mathrm{d}x+\int_{J_u}(1+ |(u^+-u^-)\odot \nu_u|)\,\mathrm{d}\mathcal{H}^{d-1}+\int_\Omega \Psi(|u|)\,\mathrm{d}x\\
& u \mapsto\int_\Omega |e(u)|^p\,\mathrm{d}x+\int_{J_u}\big(1+ G_M((u^+-u^-)\odot \nu_u)\big)\,\mathrm{d}\mathcal{H}^{d-1}+\int_\Omega \Psi(|u|)\,\mathrm{d}x\,,
\end{split}
\]
where $G_M$ is the suitable truncation of the Frobenius norm of $(u^+-u^-)\odot \nu_u$ introduced in \cite{DalMOT},  see \eqref{eq: first} above,  and   $\Psi:[0,+\infty) \to [0,+\infty)$   is a coercive function needed for applying the compactness theorem. We consider this as being a major outcome of our results.  Let us mention that we  did not address in this paper the possibility of working in the  larger  space $GSBD^p_\infty$ introduced  recently  in \cite{Crismale-Friedrich}, building on a  recent compactness result by {\sc Chambolle and Crismale} \cite{Crismale}. This  would allow us to drop the additional term $\int_\Omega \Psi(|u|)\,\mathrm{d}x$, in favor of Dirichlet boundary conditions. In  any case, this is a further interesting issue which we plan to address in the future.

   \textbf{Organization of the paper and notation:}  The paper is organized as follows. In Section \ref{sec: general} we introduce our setting,  define $BD$-ellipticity,  and prove the lower semicontinuity of $BD$-elliptic functionals in $PR(\Omega)$.  Here, we also address the problem of relaxation.   Section \ref{sec: symmetric joint convexity} is devoted to the notion of symmetric joint convexity. There, we prove $BD$-ellipticity and lower semicontinuity in $PR(\Omega)$ of the associated energies. In Section \ref{sec: examples},  we discuss   the aforementioned relevant examples of functionals which comply with our assumptions, as well as the comparison between the notions of $BV$- and $BD$-ellipticity. Finally, in Section \ref{sec: GSBD} we prove that surface energies associated to symmetric jointly convex integrands are lower semicontinuous in $GSBD^p$.

We close the introduction by fixing notations. Throughout the paper,  $\Omega \subset \R^d$ is open, bounded with Lipschitz boundary. Let $\mathcal{A}(\Omega)$ be the  family of open subsets of $\Omega$, and  let $\mathcal{A}_0(\Omega) \subset \mathcal{A}(\Omega)$ be the subset of sets with regular boundary.    The notations $\mathcal{L}^d$ and $\mathcal{H}^{d-1}$ are used for  the Lebesgue measure, and the $(d-1)$-dimensional Hausdorff measure in $\mathbb{R}^d$, respectively. For an $\mathcal{L}^d$-measurable set $E\subset\mathbb{R}^d$,  the symbol $\chi_E$ denotes its characteristic function. For $A,B \in \mathcal{A}(\Omega)$ with $\overline{B} \subset A$, we write $B \subset \subset A$. By $A \triangle B = (A\setminus B) \cup (B \setminus A)$ we denote the symmetric difference of two sets.   The symbol $B_R(x)$ denotes a ball of radius $R$ centered at $x$.

 Components of vectors $\mu \in \R^d$ are generally indicated by $\mu_k$, $k=1,\ldots,d$.  We write $\langle \mu,\mu' \rangle$ for the scalar product of two vectors $\mu,\mu' \in \R^d$.  The space of symmetric and skew-symmetric matrices is denoted by $\R^{d\times d}_{\rm sym}$ and $\R^{d\times d}_{\rm skew}$, respectively, while the identity in $\R^{d \times d}$ is indicated by ${\rm Id}$.  The Frobenius norm of a matrix $A \in \R^{d \times d}$ is indicated by $|A|$, and $\Vert A \Vert$ denotes the operator norm.  The scalar product of two matrices $A,B \in \R^{d \times d}$ is indicated by $A:B$.   The symbol $\mathbb{S}^{d-1}$ stands for the unit sphere in $\mathbb{R}^d$. For $\nu \in \mathbb{S}^{d-1}$, we denote   by  $ Q_\rho^\nu  \subset \R^d$ the $d$-dimensional cube, centered in the origin, with sidelength $\rho>0$, and two faces orthogonal to $\nu$.


\section{$BD$-ellipticity and lower semicontinuity}\label{sec: general}

In this section, we consider functionals defined on piecewise rigid functions, and we characterize lower semicontinuity in terms of an integral condition that we call \emph{$BD$-ellipticity}.   Afterwards, we also address the problem of relaxation. 

\subsection{Definitions}\label{sec: def}

In this subsection, we collect the basic definitions. 

\textbf{Function spaces:} First, we define the space of \emph{piecewise rigid functions} by
\begin{align}\label{eq: PR-def}
PR(\Omega) := & \big\{ u\colon \Omega \to   \R^d \text{ $\mathcal{L}^d$-measurable: } \   u(x) = \sum\nolimits_{k\in \N} (Q_k\, x + b_k) \chi_{P_k}(x) \ \  \text{ $\forall x \in \Omega$}, \notag \\
 &  \ \ \text{where }  Q_k \in \R^{d\times d}_{\rm skew}, \, b_k \in \R^d, \text{ and }  (P_k)_k  \text{ is a Caccioppoli partition of } \Omega \big\}. 
\end{align}
 We will sometimes use the shorthand  $a_{Q,b}(x):=Qx+b$  with $Q\in \R^{d\times d}_{\rm skew}$ and $b \in \R^d$.  It follows from the properties of Caccioppoli partitions,   see \cite[Section 4.4]{Ambrosio-Fusco-Pallara:2000},   that for each $u \in PR(\Omega)$ we have that $\mathcal{H}^{d-1}(J_u \setminus \bigcup_k \partial^* P_k) = 0$ and thus $\mathcal{H}^{d-1}(J_u)  < + \infty$.  We also note that the representation in \eqref{eq: PR-def} can always be chosen in such a way that also $\mathcal{H}^{d-1}(J_u \triangle ( \bigcup_k \partial^* P_k \setminus \partial \Omega) ) = 0$ holds, cf.\  \cite[Equation (3.2)]{FM}. In the following, we say that a sequence $(u_h)_h$ converges to $u$ in $PR(\Omega)$ if $\sup_h \mathcal{H}^{d-1}(J_{u_h})<+\infty$ and $u_h \to u$ in measure on $\Omega$.   

If $u\in PR(\Omega)$ has the form $u = \sum\nolimits_{k\in \N} b_k \chi_{P_k}$, i.e., $Q_k=0$  for all $k\in\mathbb{N}$  in representation \eqref{eq: PR-def},  then  $u\in PC(\Omega)$, where $PC(\Omega)\subset PR(\Omega)$ denotes the subspace of \emph{piecewise constant functions}.   \BBB  We also note that in \cite[Theorem 2.2]{Conti-Iurlano:15} the inclusion $PR(\Omega) \subset (GSBV(\Omega))^d$ has been shown. (See \cite[Section 4.5]{Ambrosio-Fusco-Pallara:2000} for the definition and properties of the latter function space.) \EEE   

\textbf{$BV$- and $BD$-ellipticity:} Given  $\mathcal{L}^d$-measurable  functions   $f\colon\R^d\times\R^d\times \mathbb{S}^{d-1}\mapsto [0,+\infty)$, we consider integral functionals  $\mathcal{F}\colon PR(\Omega) \times  \mathcal{A}(\Omega)   \to  [0,+\infty)  $ of the form
\begin{align}\label{eq: basic problem}
\mathcal{F}(u,A) = \int_{J_u\cap  A} f(u^+,u^-,\nu_u)\, {\rm d}\mathcal{H}^{d-1}\quad \quad \forall \, u \in PR(\Omega), \quad \quad \forall\,A\in\mathcal{A}(\Omega),
\end{align}
where $u^+, u^-$ represent the approximate  one-sided traces of $u$ on $J_u$,  $\nu_u$ denotes  a  unit normal to the jump (i.e., a normal to the interface), and $f$ represents an interfacial energy density.  (In the following, we will sometimes also write $[u]:= u^+ - u^-$.)  We often write $\mathcal{F}(u)$ instead of $\mathcal{F}(u,\Omega)$ if no confusion arises.   We assume that the functions $f$ satisfy the symmetry condition 
\begin{align}\label{eq: symmetry}
f(i,j,\nu) = f(j,i,-\nu)\quad \quad \text{for all } (i,j,\nu) \in \R^d\times \R^d \times \mathbb{S}^{d-1}.
\end{align}

We first restrict the functionals onto the subspace $PC(\Omega)$ and recall the notion of $BV$-ellipticity introduced in \cite{AmbrosioBraides2}. Fix  $(i,j,\nu)\in \R^d\times\R^d\times \mathbb{S}^{d-1}$ with $i\neq j$, and define the function $u_{i,j,\nu}\colon Q_1^\nu \rightarrow \R^d$ by
\begin{align}\label{u_0}
u_{i,j,\nu}(x):=
\begin{cases}
i &\langle x, \nu\rangle  > 0,\\
j &\langle x , \nu \rangle \leq 0.
\end{cases}
\end{align}
Then, we say that $f$ is $BV$-elliptic if 
$$
\int_{J_v}f(v^+,v^-,\nu_v)\, {\rm d}\mathcal{H}^{d-1} \geq f(i,j,\nu) \quad \quad
$$
for any $v \in PC(Q^\nu_1)$ such that $\{u_{i,j,\nu}\neq v  \}\subset\!\subset Q^\nu_{1}$ and for any triple $(i,j,\nu)$ in the domain of $f$  with $i \neq j$.  This notion is necessary and sufficient for lower semicontinuity in the space $PC(\Omega)$,  whenever $f$ is continuous and bounded, see \cite[Theorem 5.14]{Ambrosio-Fusco-Pallara:2000}.  It plays the analogous role of quasiconvexity for integral functionals defined on Sobolev spaces.  In particular,  we recall from \cite[Theorem 5.11]{Ambrosio-Fusco-Pallara:2000} the following two necessary conditions for lower semicontinuity  for continuous densities:  
\begin{itemize}
\item[i)](\emph{subadditivity}) for any $\rho\in  \mathbb{S}^{d-1} $ we have 
\begin{align*}
f(i,j,\rho)\leq f(i,k,\rho)+f(k,j,\rho)\quad \forall\,i,j,k\in \R^d;
\end{align*}
\item[ii)](\emph{convexity}) for any $i,j\in\R^d$, the function $\rho\mapsto f(i,j,\rho)$ is convex in $\R^d$.
\end{itemize}
 (We point out that, strictly speaking, \cite[Theorem 5.11, Theorem 5.14]{Ambrosio-Fusco-Pallara:2000} have been shown only when $i$ and $j$ are chosen from a countable, bounded subset of $\R^d$.  An inspection of the proofs, however,  shows  that  the results can be generalized to the whole $\R^d$.) 

 We now introduce a similar notion for functionals defined on $PR(\Omega)$ which we call \emph{$BD$-ellipticity}. 

\begin{definition}
Let $f:\R^d\times\R^d\times \mathbb{S}^{d-1}\mapsto [0,+\infty)$ be  an $\mathcal{L}^d$-measurable  function. We say that $f$ is \emph{$BD$-elliptic} if 
\begin{align}\label{eq: def-BD}
\int_{J_v}f(v^+,v^-,\nu_v)\, {\rm d}\mathcal{H}^{d-1} \geq f(i,j,\nu)
\end{align}
for any  $v \in PR(Q^\nu_1)$ such that $\{u_{i,j,\nu}\neq v  \}\subset\!\subset Q^\nu_{1}$ and for any triple $(i,j,\nu)$ in the domain of $f$  with $i \neq j$. 
\end{definition}
We have chosen this name  in analogy  to $BV$-ellipticity to highlight that $PR(\Omega)$ is related to the space $BD$, whereas the space $PC(\Omega)$ is related to the theory of $BV$-functions.    We also remark that inequality  \eqref{eq: def-BD} needs to hold only  for $i \neq j$, as the values $f(i,i, \nu)$ clearly do not  matter for the functionals in \eqref{eq: basic problem}. 

We observe that every $BD$-elliptic function is of course also $BV$-elliptic since $PC(\Omega)\subset PR(\Omega)$.  In particular, the two properties stated above (subadditivity and convexity) are necessary for $BD$-ellipticity.   The reverse implication does not hold,  i.e., $BV$- and $BD$-ellipticity are really different notions.  In fact, functions $f\colon \R^d\times \R^d\times \mathbb{S}^{d-1}\to [0,+\infty)$ of the form
\begin{align}\label{eq: BV class}
f(i,j,\nu) = \theta(i,j) \,\psi(\nu)
\end{align}
are $BV$-elliptic if $\theta\colon\R^d\times \R^d \to[0,+\infty)$ is a pseudo-distance (i.e., positive, symmetric obeying the triangle inequality) and $\psi\colon\R^d \to [0,+\infty)$ is  even, positively $1$-homogeneous, and convex. We refer to \cite[Example 5.23]{Ambrosio-Fusco-Pallara:2000} for details.  On the other hand, as we will detail below in Examples \ref{counterexample}--\ref{counterexample2}, these functions are in general not $BD$-elliptic if they are \emph{anisotropic}, i.e., if  $\theta( i, j)$ oscillates on $\lbrace (i,j)\colon |i-j| = const.\rbrace$ or $\psi(\nu)$ oscillates for $\nu \in \mathbb{S}^{d-1}$.

\subsection{Characterization of lower semicontinuity}
Recall that $BV$-ellipticity has been identified as a necessary and sufficient condition for lower semicontinuity of functionals defined on $PC(\Omega)$. In this subsection, we establish a corresponding result in $PR(\Omega)$ in terms of $BD$-ellipticity. To this end, we need to assume a slightly stronger continuity condition of the integrands $f$,  namely uniform continuity in the first two variables:  there exists an increasing modulus of continuity $\sigma:[0,+\infty) \to  [0,+\infty)  $ with $\sigma(0)=0$ such that for any $(i_1,j_1), (i_2,j_2) \in \R^d \times \R^d$  we have
\begin{align}\label{eq: continuity}
 |f(i_1,j_1,\nu)  - f(i_2,j_2,\nu)| \le   \sigma\big(|i_1-i_2| + |j_1-j_2|\big).  
\end{align}

\begin{theorem}[Characterization of lower semicontinuity]\label{thm: LSC iff BD-ell + bounded}
Let $f\colon\R^d\times\R^d\times \mathbb{S}^{d-1}\to [0,+\infty)$ be a bounded, continuous  function satisfying \eqref{eq: symmetry} and \eqref{eq: continuity}. Then, $\mathcal{F}$ defined in (\ref{eq: basic problem}) is lower semicontinuous in $PR(\Omega)$ if and only if $f$ is $BD$-elliptic. 
\end{theorem}

In order to prove the above result, we need the following fundamental estimate slightly adapted for our purposes, see \cite[Lemma 4.5]{FM}. 

\begin{lemma}\label{lemma: fundamental estimate2-new}
Let $\eta >0$ and  $A', A, B \in \mathcal{A}_0(\Omega)$ with $A' \subset \subset  A$. Let $f:\R^d\times\R^d\times \mathbb{S}^{d-1}\to [0,+\infty)$ be a measurable, bounded function satisfying  \eqref{eq: symmetry}, \eqref{eq: continuity},  and $\inf f >0$.   Let $\psi\colon [0,+\infty) \to [0,+\infty)$ be continuous and strictly increasing with $\psi(0) = 0$. Then, there exist a function $\Psi\colon PR(B) \to (0,+\infty]$  and a  lower semicontinuous function  $\Lambda: PR(A) \times PR(B) \to [0,+\infty]$ satisfying 
\begin{align}\label{eq: Lambda0}
\Lambda(z_1,z_2) \to 0 \ \text{ whenever } \ \int_{(A\setminus A')\cap B} \psi(|z_1 - z_2|) \to 0
\end{align}
such that for all $u \in PR(A)$ and  $v \in PR(B)$ satisfying the condition 
\begin{align}\label{eq: extra condition}
\Lambda(u,v) \le \Psi(v)
\end{align}  
there exists a function $w \in PR(A' \cup B)$ such that
\begin{align}\label{eq: assertionfund-newnew}
{\rm (i)}& \ \ \mathcal{F}(w, A' \cup B) \le  \mathcal{F}(u,A)  + \mathcal{F}(v, B)  + \big(M + \mathcal{F}(u,A)  + \mathcal{F}(v, B) \big) \big(2\eta +  M'  \sigma (\Theta(u,v))\big), \notag \\ 
{\rm (ii)}& \ \   \Vert \min\lbrace |w - u|, |w-v| \rbrace \Vert_{L^\infty(A' \cup B)} \le \Theta(u,v),\notag\\
{\rm (iii)} & \ \ w = v \text{ on } B\setminus A.
\end{align}
Here,  $\mathcal{F}$ is of the form \eqref{eq: basic problem}, $\sigma$ is given in \eqref{eq: continuity}, and  $M,M'>0$  as well as $\Theta\colon PR(A) \times PR(B) \to [0,+\infty]$  are independent of $u$ and $v$. Moreover,  $M$ is also independent of $\eta$ and   $\Theta$   is a  lower  semicontinuous  function satisfying 
\begin{align}\label{eq: Theta convergence}
\Theta(z_1,z_2) \to 0 \ \text{ whenever } \ \int_{(A\setminus A')\cap B} \psi(|z_1 - z_2|) \to 0\,.
\end{align}
\end{lemma}

In the above result, we follow the convention that $u \in PR(A)$ and $v \in PR(B)$ are extended by $u = 0$ and $v = 0$ outside of $A$ and $B$, respectively.  Condition \eqref{eq: extra condition} is necessary to ensure \eqref{eq: assertionfund-newnew}(iii). As detailed  in \cite[Lemma 4.1]{FM},   \eqref{eq: extra condition} can be removed at the expense of dropping also \eqref{eq: assertionfund-newnew}(iii).

\begin{proof}
We briefly explain how the result follows from \cite[Lemma 4.5]{FM}. The functional $\mathcal{F}$ satisfies (H$_1$) since $f$ is measurable and (H$_3$) holds due to the integral representation \eqref{eq: basic problem}.  Property (H$_4$) follows from the fact that $f$ is bounded and satisfies $\inf f >0$,  where we set $\alpha = \inf f>0$ and   $\beta =  \sup f< \infty$.   Property \eqref{eq: continuity} implies (H$_5^\prime$). Condition \eqref{eq: extra condition}  is equivalent to  \cite[Equation (4.6)]{FM} for given $\delta>0$,  $M_1 \ge 0$,   and $\Psi(v) := \frac{1}{M_1} \Phi(A',A'\cup B; v|_{B \setminus \overline{A'}},\delta)  \in (0,+\infty]$.  Then,  \cite[Equation (4.7)]{FM} implies  \eqref{eq: assertionfund-newnew}, where we set  $M := \mathcal{H}^{d-1}(\partial A' \cup \partial A \cup \partial B)$ and  $M'  := M_2$.  
\end{proof}

\begin{proof}[Proof of Theorem \ref{thm: LSC iff BD-ell + bounded}]
Our proof is in the spirit of \cite[Theorem 5.14]{Ambrosio-Fusco-Pallara:2000} with the essential difference that in the implication ``$BD$-ellipticity implies lower semicontinuity'' we replace the lemma of joining two functions, see \cite[Lemma 5.15]{Ambrosio-Fusco-Pallara:2000}  or \cite[Lemma 4.4]{AmbrosioBraides},  by our fundamental estimate stated in Lemma \ref{lemma: fundamental estimate2-new}. We show the two directions separately. 

\noindent \emph{Step 1: Lower semicontinuity implies $BD$-ellipticity.} The argument is very similar to the one used in \cite[Theorem 5.14]{Ambrosio-Fusco-Pallara:2000}, and we therefore only sketch it. Up to a rescaling and a translation of $Q^\nu_1$, we may assume that $Q^\nu_1 \subset \Omega$. Consider $i,j\in\R^d$  with $i \neq j$,  $\nu \in \mathbb{S}^{d-1}$, and some $v \in PR(Q^\nu_1)$ with $\{u_{i,j,\nu}\neq v  \}\subset\!\subset Q_{1}^\nu$. For  each $h \in \N$,   we define $u_h \in PR(\Omega)$ by $u_h = u_{i,j,\nu}$ on $\Omega \setminus Q^\nu_1$ and inside $Q^\nu_1$ we set 
\begin{align*}
u_h(x) = \begin{cases}
i & \text{if } \langle x,\nu \rangle > 1/h,\\
v(h(x-x_n)) & \text{if }   0<\langle x,\nu\rangle < 1/h  \text{ and } x \in Q_n,\\
j & \text{if } \langle x,\nu \rangle <0,
\end{cases}
\end{align*}
where $(Q_n)_n$ denotes a partition of the set $\lbrace x\in Q^\nu_1\colon 0<\langle x,\nu \rangle < 1/h \rbrace$ consisting of $h^{d-1}$ cubes with sidelength $1/h$,  and $x_n$ indicates the center of $Q_n$.  As $\{u_{i,j,\nu}\neq v  \}\subset\!\subset Q_{1}^\nu$, we find by a scaling argument 
\begin{align*}
\mathcal{F}(u_h)& \le  \sum_{n=1}^{h^{d-1}} \int_{J_{u_h} \cap Q_n} f(u_h^+,u^-_h,\nu_{u_h}) \, {\rm d}\mathcal{H}^{d-1} + \Vert f \Vert_\infty\mathcal{H}^{d-1}(J_{u_h} \cap \partial Q^\nu_1) + \mathcal{H}^{d-1}(\Gamma) f(i,j,\nu)\\& \le \int_{J_{v} \cap Q^\nu_1} f(v^+,v^-,\nu_{v}) \, {\rm d}\mathcal{H}^{d-1}+ C\Vert f\Vert_\infty/h + \mathcal{H}^{d-1}(\Gamma) f(i,j,\nu),
\end{align*}
where $\Gamma := \lbrace x\colon \langle x,\nu \rangle = 0 \rbrace \cap (\Omega \setminus Q_1^\nu)$  and $C>0$ is a universal constant.  Since $\mathcal{L}^d(\lbrace u_h \neq u_{i,j,\nu} \rbrace) \le 1/h$, we find $u_h \to  u_{i,j,\nu}  $ in measure on $\Omega$. Therefore, by the lower semicontinuiy of $\mathcal{F}$ we conclude
 \begin{align*}
 \int_{J_{v} \cap Q^\nu_1} f(v^+,v^-,\nu_{v}) \, {\rm d}\mathcal{H}^{d-1} & \ge \liminf_{h\to \infty} \mathcal{F}(u_h) - \mathcal{H}^{d-1}(\Gamma) f(i,j,\nu) \ge  \mathcal{F}(u_{i,j,\nu}) - \mathcal{H}^{d-1}(\Gamma) f(i,j,\nu) \\ & = \big(\mathcal{H}^{d-1}(J_{u_{i,j,\nu}} \cap \Omega) - \mathcal{H}^{d-1}(\Gamma) \big)f(i,j,\nu) = f(i,j,\nu).
 \end{align*}       
 This shows that $f$ is $BD$-elliptic.

\noindent \emph{Step 2: $BD$-ellipticity implies lower semicontinuity.} We detail this step only in the special case  $\Omega=Q_1^\nu$ for the special limiting function $u_{i,j,\nu}$ for some $i,j\in\R^d$  with $i\neq j$  and $\nu \in \mathbb{S}^{d-1}$. The general case follows by  standard covering and blow up arguments. We refer to Step 2 and Step 3 in the proof of \cite[Theorem 5.14]{Ambrosio-Fusco-Pallara:2000} for details.

 Let $(u_h)_{h}\subset PR(Q_1^\nu)$ be a sequence converging to $u_{i,j,\nu}$  in $PR(Q^\nu_1)$. In particular, we have $\sup_h \mathcal{H}^{d-1}(J_{u_h})<+\infty$, and the boundedness of $f$ then implies    
 \begin{align}\label{eq: boundedness}
 \sup\nolimits_h\mathcal{F}(u_h,Q^\nu_1) <+\infty.
\end{align} 
 We first suppose that $\inf f>0$ and explain the small adaptions for $\inf f = 0$ at the end of the proof.  We want to construct a sequence $(w_h)_{h}\subset PR(Q^\nu_1)$ such that $\{u_{i,j,\nu}\neq w_h  \}\subset\!\subset Q_{1}^\nu$   and  such that   the energy of $w_h$ is  asymptotically controlled by the one of $u_h$. Our  strategy relies on Lemma~\ref{lemma: fundamental estimate2-new}. 
 
 To this end, we first fix  $\eta >0$, $\rho>0$, and define the sets  $B,A,A'\subset  \mathcal{A}_0(Q_1^\nu)$ by
 $A' = Q^\nu_{1-2\rho}$, $A = Q^\nu_{1-\rho}$, and $B = Q^\nu_1 \setminus Q^\nu_{1-3\rho}$. Note that $A' \cup B = Q^\nu_1$. In order to  apply Lemma \ref{lemma: fundamental estimate2-new} for $u =  u_h|_A  $ and $v =  u_{i,j,\nu}|_B$, we need to check \eqref{eq: extra condition}. As $u_h \to u_{i,j,\nu}$ in measure on $Q^\nu_1$, we clearly get  
\begin{align}\label{eq: measure conv}
\int_{Q_1^\nu} \psi (|u_h - u_{i,j,\nu}|)\to 0
\end{align}
  as $h \to \infty$, where $\psi\colon [0,+\infty) \to [0,+\infty)$ is defined by $\psi(t) := t/(1+t)$  for $t \ge 0$.   Therefore, $\Lambda(u_h|_A,u_{i,j,\nu}|_B) \to 0$ by  \eqref{eq: Lambda0}. Consequently, there holds  $ \Lambda(u_h|_A,u_{i,j,\nu}|_B) \le \Psi(u_{i,j,\nu}|_B)$  for all $h$ sufficiently large and thus \eqref{eq: extra condition} holds.  

We apply Lemma \ref{lemma: fundamental estimate2-new} for $u = u_h|_A$ and $v =  u_{i,j,\nu}|_B$, and obtain $w_h \in PR(Q^\nu_1)$ such that by  \eqref{eq: assertionfund-newnew} there holds $w_h = u_{i,j,\nu}$ on $Q^\nu_1 \setminus Q^\nu_{1-\rho}$  and 
\begin{align}\label{eq: ...<....+I_1+I_2-XXX}
\mathcal{F}(w_h,Q^\nu_1)&\leq \mathcal{F}\big(u_h, Q^\nu_{1-\rho}\big)+ \mathcal{F}\big(u_{i,j,\nu},Q^\nu_{1} \setminus Q^\nu_{1-3\rho}\big) + I^{h, \eta}_1 +I^{h,\eta}_2,
\end{align}
where for shorthand we have set 
\begin{align*}
I^{h,\eta}_1&= 2\eta\left(M+ \mathcal{F}(u_h, A)+ \mathcal{F}(u_{i,j,\nu},B) \right), \\
I^{h,\eta}_2&=  M'  \sigma\big(  \Theta (u_h|_A, u_{i,j,\nu}|_B)\big)  \left( M+ \mathcal{F}(u_h, A)+ \mathcal{F}(u_{i,j,\nu},B) \right).
\end{align*}
Since $f$ is nonnegative and $BD$-elliptic, and there holds $\{u_{i,j,\nu}\neq w_h  \}\subset\!\subset Q_{1}^\nu$, we get by \eqref{eq: ...<....+I_1+I_2-XXX}
\begin{align}\label{eq: ...<....+I_1+I_2}
\mathcal{F}(u_{i,j,\nu},Q_1^\nu)\le  \mathcal{F}(w_h, Q^\nu_1) \le \mathcal{F}\big(u_h, Q^\nu_1\big)+ \mathcal{F}\big(u_{i,j,\nu}, Q^\nu_{1} \setminus Q^\nu_{1-3\rho}\big) + I^{h,\eta}_1 +I^{h,\eta}_2.
\end{align}
By  $\sigma(0)=0$,  \eqref{eq: Theta convergence}, \eqref{eq: boundedness}, and \eqref{eq: measure conv} we obtain $\lim_{h \to \infty} I^{h,\eta}_2 = 0$. This along with \eqref{eq: ...<....+I_1+I_2} implies
\begin{align*}
\mathcal{F}(u_{i,j,\nu},Q_1^\nu)&\leq \liminf_{h\rightarrow \infty} \mathcal{F}(u_h, Q_1^\nu)+ \mathcal{F}\big(u_{i,j,\nu}, Q^\nu_{1} \setminus Q^\nu_{1-3\rho}\big) +\sup\nolimits_h I^{h,\eta}_1.
\end{align*}
By \eqref{eq: boundedness}  and the fact that $M$ is independent of $\eta$  we get $\lim_{\eta \to 0} (\sup_h I_1^{h,\eta})=0$. Thus, passing to the limits $\eta,\rho \to 0$, we obtain the desired estimate
\begin{align*}
\mathcal{F}(u_{i,j,\nu},Q_1^\nu)&\leq \liminf_{h\rightarrow \infty} \mathcal{F}(u_h, Q_1^\nu).
\end{align*}
This concludes the proof  in the case $\inf f>0$.  If $\inf f= 0$ instead, we consider densities $f_\eps = f+ \eps$ for arbitrary $\eps>0$. As $f$ is $BD$-elliptic and the constant function with value $\eps$ is $BD$-elliptic (see e.g.\ Proposition \ref{prop: only nu}   below), we see that also $f_\eps$ is $BD$-elliptic.
Then, the functional $\mathcal{F}_\eps$ with density $f_\eps$ is lower semicontinuous and we obtain
\begin{align*}
\mathcal{F}(u) \le \mathcal{F}_\eps(u) \le \liminf_{h \to \infty} \mathcal{F}_\eps(u_h) \le \liminf_{h \to \infty} \mathcal{F}(u_h)+\eps \sup\nolimits_h \mathcal{H}^{d-1}(J_{u_h}).
\end{align*}
We conclude the proof by passing to $\eps \to 0$  and using  \eqref{eq: boundedness}. 
\end{proof}

\begin{remark}\label{rem: just bounded}
For later purposes, we note that  in Step 1 of the  proof we only used the boundedness of $f$ but not its continuity. In other words, lower semicontinuity in $PR(\Omega)$ implies $BD$-ellipticity for bounded, measurable functions $f\colon\R^d\times\R^d\times \mathbb{S}^{d-1}\to [0,+\infty)$ satisfying \eqref{eq: symmetry}. 
\end{remark}

We now drop the condition that $f$ is a bounded function, and obtain the following  two  corollaries. 

\begin{corollary}[Lower semicontinuity for unbounded functions]\label{cor: unbounded}
If $f$ is  a continuous,  $BD$-elliptic function  satisfying \eqref{eq: symmetry},  then $\mathcal{F}$ defined in (\ref{eq: basic problem})  is lower semicontinuous along sequences  converging  in $PR(\Omega)$  which are bounded in $L^\infty(\Omega;\R^d)$.  
\end{corollary}
\begin{proof}
Given $(u_h)_{h}$ with $M:= \sup_{h\in\mathbb{N}}\Vert u_h \Vert_\infty <+\infty$, we choose a bounded, continuous function  $\tilde{f}\colon\R^d\times\R^d\times \mathbb{S}^{d-1}\to [0,+\infty)$ such that $f(i,j,\nu) = \tilde{f}(i,j,\nu)$ whenever $|i|,|j| \le M$ and $\nu \in \mathbb{S}^{d-1}$. By uniform continuity on compact sets, this can be achieved in such a way that $\tilde{f}$ satisfies also \eqref{eq: continuity}. The statement now follows from Theorem \ref{thm: LSC iff BD-ell + bounded} noting that the sequence  of energies  $(\mathcal{F}(u_h))_h$ remains unchanged when $f$ is replaced by $\tilde{f}$ in \eqref{eq: basic problem}.    
\end{proof}

\begin{corollary}[Supremum of bounded $BD$-elliptic functions]\label{cor: f=sup fh is BD and LSC}
Let $f_h\colon\R^d\times\R^d\times \mathbb{S}^{d-1}\rightarrow [0,+\infty)$, $h\in\mathbb{N}$, be a sequence of continuous, bounded, and $BD$-elliptic functions satisfying   \eqref{eq: symmetry} and  \eqref{eq: continuity}. Suppose $f(i,j,\nu):=\sup_{h\in\mathbb{N}}f_h(i,j,\nu)<+\infty$ for all $i,j\in \R^d$,  $i \neq j$,  and $\nu \in \mathbb{S}^{d-1}$. Then, $f$ is $BD$-elliptic and the corresponding functional $\mathcal{F}$ defined in (\ref{eq: basic problem}) is lower semicontinuous on $PR(\Omega)$.
\end{corollary}

To prove the above corollary, let us recall the following lemma (see, e.g., \cite[Lemma 15.2]{Braides:02}).
\begin{lemma}\label{lem: localization}
Let $\Omega$ be an open subset of $\R^d$. Let $\Lambda$ be  a  set function defined on $\mathcal{A}(\Omega)$, which is supperadditive on open sets with disjoint compact closure, i.e., $ \Lambda  (U \cup V) \geq \Lambda(U) + \Lambda(V)$ whenever $U$, $V \subset \subset \Omega$ and $\ove U \cap \ove V = \emptyset$.  Let $\lambda$ be a positive measure on $\Omega$, and let $(\varphi_h)_{h}$ be a sequence of nonnegative Borel functions on $\Omega$ such that  $\Lambda(A) \geq \int_A \varphi_h \,\mathrm{d}\lambda$ for every $A\in \mathcal{A}(\Omega)$ and  $h \in \N$. Then, $\int_{A}\sup_h\varphi_h {\rm d}\lambda\leq \Lambda(A)$ for every $A\in \mathcal{A}(\Omega)$.
\end{lemma}

\begin{proof}[Proof of Corollary \ref{cor: f=sup fh is BD and LSC}] We first prove $BD$-ellipticity and then lower semicontinuity.
 
\noindent \emph{Step 1: $BD$-ellipticity.} We first show that $f$ is $BD$-elliptic. Fix a triple $(i,j,\nu)\in \R^d\times \R^d \times \mathbb{S}^{d-1}$  with $i \neq j$.  Let $u\in PR(Q_1^\nu)$ be such that $\{ u\neq u_{i,j,\nu} \}\subset \!\subset Q_1^\nu$, where  $u_{i,j,\nu}$ is defined in (\ref{u_0}). Then, since each  $f_h$ is $BD$-elliptic, we get 
\begin{align*}
\mathcal{F}(u,Q_1^\nu)&= \int_{J_u}f(u^+,u^-,\nu_u)\,{\rm d}\mathcal{H}^{d-1}=\int_{J_{u}}\sup_{h\in\mathbb{N}}f_h(u^+,u^-,\nu_u)\,{\rm d}\mathcal{H}^{d-1}\\&\geq \sup_{h\in\mathbb{N}}\int_{J_{u}}f_h(u^+,u^-,\nu_u)\,{\rm d}\mathcal{H}^{d-1}\geq \sup_{h\in\mathbb{N}}  f_h(i,j,\nu) = f(i,j,\nu).
\end{align*}
\emph{Step 2: Lower semicontinuity.} Consider a sequence $(u_n)_n\subset PR(\Omega)$ and $u\in PR(\Omega)$ such that $u_n\to u $   in $PR(\Omega)$  as $n \to \infty$. Our goal is to show 
\begin{align}\label{eq: thesis of cor f=sup fh is BD and LSC}
\liminf_{n\rightarrow \infty}\mathcal{F}(u_n,\Omega)\geq \mathcal{F}(u,\Omega).
\end{align}
In view of Theorem \ref{thm: LSC iff BD-ell + bounded}, the functional with integrand $f_h$ is lower semicontinuous  for every $h\in\mathbb{N}$. Therefore, we get for every $U \in \mathcal{A}(\Omega)$
\begin{align}\label{eq: 1}
\liminf_{n\rightarrow \infty} \int_{J_{u_n} \cap U}f(u_n^+,u_n^-,\nu_{u_n})\,{\rm d}\mathcal{H}^{d-1}&\geq \liminf_{n\rightarrow \infty} \int_{J_{u_n}\cap U}f_h(u_n^+,u_n^-,\nu_{u_n})\,{\rm d}\mathcal{H}^{d-1}\notag \\ &\geq \int_{J_{u}\cap U}f_h(u^+,u^-,\nu_u)\,{\rm d}\mathcal{H}^{d-1}.
\end{align} 
We define the superadditive function $\Lambda\colon \mathcal{A}(\Omega)\rightarrow [0,+\infty)$ by
\begin{align*}
\Lambda (U):= \liminf_{n\rightarrow \infty}\int_{J_{u_n}\cap U}f(u^+_n,u^-_n,\nu_{u_n})\,{\rm d}\mathcal{H}^{d-1}
\end{align*}
 for each $U \in \mathcal{A}(\Omega)$.  Thus, by (\ref{eq: 1}) we obtain 
\begin{align*}
\Lambda (U) \geq \int_{J_u\cap U} f_h (u^+,u^-,\nu_u) \, {\rm d}\mathcal{H}^{d-1}
\end{align*}
for all $U\in\mathcal{A}(\Omega)$ and all $h\in\mathbb{N}$. By applying Lemma \ref{lem: localization} we get that
$$
\Lambda (U) \geq \int_{J_u\cap U} \sup_{h\in\mathbb{N}}f_h (u^+,u^-,\nu_u)\,{\rm d}\mathcal{H}^{d-1} = \int_{J_u\cap U} f (u^+,u^-,\nu_u)\,{\rm d}\mathcal{H}^{d-1}  
$$
for all $U \in \mathcal{A}(\Omega)$. For $U = \Omega$,  we obtain \eqref{eq: thesis of cor f=sup fh is BD and LSC}. This concludes the proof. 
\end{proof}

\subsection{Relaxation}
In this subsection, we address the relaxation of integral functionals of the form  \eqref{eq: basic problem}. For simplicity, we restrict our study to the class of translational invariant integrands, i.e., functions $f\colon\R^d\times \R^d \times  \mathbb{S}^{d-1}\mapsto [0,+\infty)$ satisfying 
\begin{align}\label{eq: restriction} 
f(i,j,\nu) = f(i+t,j+t,\nu)\quad \quad \text{for all } (i,j,\nu) \in \R^d\times \R^d \times \mathbb{S}^{d-1} \text{ and } t \in \R^d. 
\end{align} 
In other words,  differently from what considered so far,   functionals of the form (\ref{eq: basic problem}) depend on the two vectors $([u],\nu_u)$ rather than on the more general triple $(u^+,u^-,\nu_u)$. (For consistency, we keep the notation $(u^+,u^-,\nu_u)$ in the following.) This assumption  is due to the fact that we will use an integral representation result \cite{FM} which has been proved in this slightly more specific setting only. In \cite{FM}, however, translational invariance is assumed just to simplify the exposition, and a generalization to the general situation of  \eqref{eq: basic problem} would in principle be possible.  We note that, under \eqref{eq: restriction}, the continuity condition \eqref{eq: continuity} can be reduced to
\begin{align}\label{eq: continuity2}
|f(\xi,0,\nu)-f(\tau,0,\nu)|  \le   \sigma\big(|\xi-\tau|\big) \quad\quad \text{ for all } \xi,\tau\in\R^d, \nu \in \mathbb{S}^{d-1}.
\end{align} 
Before we come to the main result of this subsection,  we introduce a further notation: for every $u \in PR(\Omega)$ and $A \in \mathcal{A}(\Omega)$ we define 
$$
\mathbf{m}_{\mathcal{F}}(u,A) = \inf_{v \in PR(\Omega)} \  \big\{ \mathcal{F}(v,A)\colon \    \{v\neq u\}\subset\subset A  \big\}.
$$


\begin{theorem}\label{thm: don't worry be happy}
Let $f: \R^d\times \R^d \times \mathbb{S}^{d-1}\rightarrow [0,+\infty)$ be a bounded, continuous function satisfying \eqref{eq: symmetry},   \eqref{eq: restriction}, \eqref{eq: continuity2}, and $\inf f >0$. Let $\mathcal{F}$ be defined in (\ref{eq: basic problem}). Then, the relaxed functional  defined as
\begin{align}\label{relaxed functional}
\bar{\mathcal{F}}(u,A):= \inf\left\{\liminf_{h\rightarrow \infty}\mathcal{F}(u_h,A)\colon\,u_h\rightarrow u \textit{ in $PR(A)$}   \right\}
\end{align}
admits an integral representation, namely
\begin{align}\label{eq: integral representation}
\bar{\mathcal{F}}(u,A)= \int_{J_u\cap A}\varphi(u^+,u^-,\nu_u)\, {\rm d}\mathcal{H}^{d-1}\quad\quad\forall \, u \in PR(\Omega), \quad \quad \forall\,A\in\mathcal{A}(\Omega).
\end{align}
Here, the function $\varphi\colon \R^d\times \R^d \times \mathbb{S}^{d-1}\rightarrow [0,+\infty)$ is the greatest $BD$-elliptic function with $\varphi \le f$ and is characterized by 
\begin{align}\label{eq: varphi = Ef}
\varphi(i,j,\nu) = \mathbf{m}_{\mathcal{F}}\big(u_{i,j,\nu},Q^\nu_1\big) \quad \quad \text{for all } (i,j,\nu) \in  \R^d \times \R^d \times  \mathbb{S}^{d-1},\ i\neq j. 
\end{align} 
Moreover, $\varphi$ is bounded, continuous and satisfies \eqref{eq: symmetry}, \eqref{eq: restriction}, and  \eqref{eq: continuity2}.
\end{theorem}

The key ingredient is the following $\Gamma$-convergence and integral representation result, see \cite[Theorem 2.3]{FM}, which is slightly adapted for our purposes. For an exhaustive treatment of $\Gamma$-convergence we refer to \cite{Braides-Defranceschi:98, DalMaso:93}. In particular, we recall that for a constant sequence of functionals the $\Gamma$-limit is given by the lower semicontinuous envelope, cf.\ \cite[Remark 4.5]{DalMaso:93}.

\begin{lemma}[$\Gamma$-convergence and integral representation]\label{th: gamma}
  Let $(\mathcal{F}_n)_n$ be a sequence of functionals of the form  \eqref{eq: basic problem} for continuous densities $f_n \colon\R^d\times\R^d\times \mathbb{S}^{d-1}\to [0,+\infty)$ satisfying  $0<\alpha \le \inf f_n \le \sup f_n \le \beta < +\infty$, \eqref{eq: symmetry},  \eqref{eq: restriction}, and \eqref{eq: continuity2}  for the same function $\sigma$.     Then, there exists $\bar{\mathcal{F}}\colon  PR(\Omega) \times \mathcal{A}(\Omega) \to [0,+\infty)$  and a subsequence (not relabeled) such that
\begin{align}\label{eq: gamma-lim}
\bar{\mathcal{F}}(\cdot,A) =\Gamma\text{-}\lim_{n \to \infty} \mathcal{F}_n(\cdot,A) \ \ \ \ \text{with respect to convergence in measure on $A$} 
\end{align}
for all $A \in  \mathcal{A}_0(\Omega) $. Moreover, if there holds 
\begin{align}\label{eq: condition-new-new}
\limsup_{n \to \infty} \mathbf{m}_{\mathcal{F}_n}(u, B_\eps(x_0)) \le  \mathbf{m}_{\bar{\mathcal{F}}}(u, B_\eps(x_0)) \le \sup\nolimits_{0<\eps' < \eps} \,  \liminf_{n \to \infty} \mathbf{m}_{\mathcal{F}_n}(u, B_{\eps'}(x_0)) 
\end{align}
for all  $u \in PR(\Omega)$ and  each  ball   $B_\eps(x_0) \subset \Omega$,  then   $\bar{\mathcal{F}}$  admits an integral representation of the form \eqref{eq: basic problem} for a density $f$ which satisfies $\alpha \le \inf f \le \sup f \le \beta$, \eqref{eq: symmetry},  \eqref{eq: restriction}, and  \eqref{eq: continuity2}.
\end{lemma}

\begin{proof}
We briefly explain how the result follows from \cite[Theorem 2.3]{FM}. The functionals $\mathcal{F}_n$ satisfy (H$_1$) since $f_n$ are measurable and (H$_3$) holds due to the integral representation \eqref{eq: basic problem}.  Property (H$_4$) follows from $\alpha \le \inf f_n \le \sup f_n \le \beta$.   Property \eqref{eq: continuity2} implies (H$_5$). 

Then, by \cite[Theorem 2.3]{FM} we find a limiting functional $\bar{\mathcal{F}}$ satisfying {\rm (${\rm H_1}$)}--{\rm (${\rm H_5}$)} such that \eqref{eq: gamma-lim} holds and $\bar{\mathcal{F}}$ admits an integral representation. It remains to show that the corresponding density $f$ satisfies $\alpha \le \inf f \le \sup f \le \beta$, \eqref{eq: symmetry},   \eqref{eq: restriction}, and \eqref{eq: continuity2}. In fact,  \eqref{eq: symmetry} and \eqref{eq: restriction} are obvious by \cite[Equation (2.7)]{FM} and {\rm (${\rm H_4}$)} implies $\alpha \le \inf f \le \sup f \le \beta$. Finally,  by {\rm (${\rm H_5}$)} we get \eqref{eq: continuity2}.
\end{proof}

\begin{proof}[Proof of Theorem \ref{thm: don't worry be happy}]
We divide the proof into three steps. 

\noindent \emph{Step 1: Integral representation.} In this step, we prove that (\ref{eq: integral representation}) holds true. We start by applying Lemma \ref{th: gamma} on the constant sequence of functionals $\mathcal{F}_n = \mathcal{F}$ for all $n \in \N$. (Note that Lemma \ref{th: gamma} is applicable as $f$ is bounded, $\inf f >0$, as well as $f$ satisfies  \eqref{eq: symmetry}, \eqref{eq: restriction}, and \eqref{eq: continuity2}.) As $\inf f >0$, convergence of measure is equivalent to convergence in $PR(\Omega)$   for sequences of bounded energy. Therefore, we get that the functional $\bar{\mathcal{F}}$ defined in \eqref{relaxed functional} coincides with the $\Gamma$-limit given in \eqref{eq: gamma-lim}, cf.\ \cite[Remark 4.5]{DalMaso:93}.  Now, to show that $\bar{\mathcal{F}}$ admits an integral representation, it remains to check that \eqref{eq: condition-new-new} holds true. To this end, it suffices to prove that 
\begin{align}\label{eq: our (2.8)}
\mathbf{m}_{\mathcal{F}}(u,A)= \mathbf{m}_{\bar{\mathcal{F}}}(u,A) \quad\quad\forall\, u\in PR(\Omega),\quad\quad \forall\, A\in\mathcal{A}_0(\Omega).
\end{align}
Observe that  inequality ``$\geq$'' follows directly by definition of $\mathbf{m}_{\mathcal{F}}$ and by the fact that $\bar{\mathcal{F}} \le \mathcal{F}$, see \eqref{relaxed functional}. The other inequality is a  direct consequence of \cite[Lemma 6.3]{FM}.  (This result essentially relies on the fundamental estimate Lemma \ref{lemma: fundamental estimate2-new}.)  Thus, (\ref{eq: our (2.8)}) holds true. Then, Lemma \ref{th: gamma} yields that $\bar{\mathcal{F}}$ admits an integral representation. The corresponding integrand is denoted by $\varphi$ in the following.  From Lemma \ref{th: gamma} we also get that $\varphi$ is bounded and satisfies \eqref{eq: symmetry},  \eqref{eq: restriction}, and  \eqref{eq: continuity2}.

\noindent \emph{Step 2: $BD$-ellipticity and representation (\ref{eq: varphi = Ef}).} As  $\Gamma$-limit, the functional $\bar{\mathcal{F}}$ is lower semicontinuous on $PR(\Omega)$. Since $\varphi$ is also bounded and satisfies \eqref{eq: symmetry} by Step 1, we get that $\varphi$ is $BD$-elliptic, see  Remark \ref{rem: just bounded}. Now, since  $\varphi$ is $BD$-elliptic,  \eqref{eq: def-BD} implies
$$\varphi(i,j,\nu) = \mathbf{m}_{\bar{\mathcal{F}}}\big(u_{i,j,\nu},Q^\nu_1 \big) \quad \quad \text{for all } (i,j,\nu)\in \R^d\times \R^d\times \mathbb{S}^{d-1},\ i \neq j.  $$ 
This along with  \eqref{eq: our (2.8)} yields (\ref{eq: varphi = Ef}) and concludes Step 2 of the proof.

\noindent \emph{Step 3: Further properties of $\varphi$.}  To conclude the proof, it remains to show that $\varphi$ is continuous and that it is the greatest $BD$-elliptic function below $f$. In view of \eqref{eq: restriction}--\eqref{eq: continuity2}, for the continuity it suffices to check that $\nu \mapsto \varphi(i,j,\nu)$ is continuous  for fixed $i,j \in \R^d$.  As $\varphi$ is $BD$-elliptic, the mapping $\nu \mapsto \varphi(i,j,\nu)$ is convex, see Subsection \ref{sec: def}. In particular, the mapping is also continuous, as desired. 
 
Finally, we show that $\varphi$ is the greatest $BD$-elliptic function with $\varphi\le f$. First,  $\varphi\le f$ clearly follows from (\ref{eq: varphi = Ef}) by using $u_{i,j,\nu}$ as a competitor. On the other hand, let $\bar{\varphi}$ be another $BD$-elliptic function satisfying $\bar{\varphi} \le f$. Let us prove that $\bar{\varphi}\le \varphi$.  Denoting by $\mathcal{F}_{\bar{\varphi}}$ the functional in \eqref{eq: basic problem} with density $\bar{\varphi}$, we find 
$$\bar{\varphi}(i,j,\nu) =  \mathbf{m}_{\mathcal{F}_{\bar{\varphi}}}\big(u_{i,j,\nu},Q^\nu_1 \big) \quad \quad \text{for all }  (i,j,\nu) \in \R^d \times \R^d \times \mathbb{S}^{d-1}, \ i \neq j$$
 since $\bar{\varphi}$ is $BD$-elliptic. Then, $\bar{\varphi} \le f$ and  \eqref{eq: varphi = Ef} imply
$$\bar{\varphi}(i,j,\nu) =  \mathbf{m}_{\mathcal{F}_{\bar{\varphi}}}\big(u_{i,j,\nu},Q^\nu_1 \big) \le \mathbf{m}_{\mathcal{F}}\big(u_{i,j,\nu},Q^\nu_1 \big) = \varphi(i,j,\nu).$$
This shows indeed that $\varphi$ is the greatest $BD$-elliptic function with $\varphi\le f$  on $\lbrace i \neq j \rbrace$. (The values on the diagonal $\lbrace i= j \rbrace$ are irrelevant.)  
\end{proof}

%
%
%
%
%
%
%
%
%

\section{A sufficient condition for lower semicontinuity: symmetric joint convexity}\label{sec: symmetric joint convexity}

Whereas Theorem \ref{thm: LSC iff BD-ell + bounded} provides a characterization of lower semicontinuity in $PR(\Omega)$ for functionals defined in \eqref{eq: basic problem}, the drawback is that it is in general a difficult task to check whether an integrand $f$  is $BD$-elliptic or not. Therefore, we seek for a sufficient condition that (a) implies $BD$-ellipticity and  lower semicontinuity, as well as  that (b) can be checked in practice for concrete examples. For $BV$-ellipticity, this role is played by \emph{jointly convex functions}. In the setting of piecewise rigid functions, we introduce a corresponding notion that we call \emph{symmetric joint  convexity}. In this section, we prove sufficiency for lower semicontinuity. We defer important examples of symmetric jointly convex functions to Section  \ref{sec: examples} below.

We recall that a vector field $g \in C(\R^d;\R^d)$ is \emph{conservative} if there exists a potential $G \in C^1(\R^d)$ such that $\nabla G =g$. 
\begin{definition}[Symmetric joint convexity]\label{def: symm-conv}
We say that $f\colon\R^d\times \R^d\times \R^d\rightarrow [0,+\infty)$ is a \emph{symmetric jointly convex function} if 
\begin{align}\label{symmetric joint convexity}
f(i,j,\nu)=\sup_{h\in\mathbb{N}}\left\langle g_h(i)-g_h(j),\, \nu\right\rangle \quad \quad \text{for all } (i,j,\nu) \in \R^d\times \R^d\times \R^d \quad  \text{with } i \neq j,
\end{align}
where $g_h\colon \R^d\to \R^d$ is a  uniformly continuous, bounded, and  conservative vector field for every $h\in\mathbb{N}$.  
\end{definition}
 The notion is related to the class of \emph{jointly convex functions}, see \cite[Definition 5.17]{Ambrosio-Fusco-Pallara:2000}, which constitutes an important class of $BV$-elliptic functions. The essential difference in our definition is that we require the vector fields to be \emph{conservative}. This additional property is  instrumental to deal with functions   for which only the symmetric part of the gradient can be controlled.  We point out  that the definition directly implies that $f$ as in \eqref{symmetric joint convexity} satisfies  \eqref{eq: symmetry}.   Before we proceed with the main statement of this section, we remark that the functions $(g_h)_h$ can be approximated by more regular functions.
 
 \begin{remark}\label{rem: symm-conv}
We will sometimes approximate functions of the kind \eqref{symmetric joint convexity} with the supremum of more regular fields, which belong to $ C^1(\R^d;\R^d) \cap W^{1,\infty}(\R^d;\R^d)$. In fact, each   uniformly continuous, bounded, and  conservative vector field can be approximated uniformly by a conservative vector field  in $C^1(\R^d;\R^d) \cap W^{1,\infty}(\R^d;\R^d)$. (This follows by approximating the corresponding potential.) Therefore, for a given $f$ as in  \eqref{symmetric joint convexity} and each $\varepsilon >0$ we can find a sequence $(g^{\varepsilon}_h)_h \subset  C^1(\R^d;\R^d) \cap W^{1,\infty}(\R^d;\R^d)$ such that
\begin{equation}\label{eq: approximation}
f(i,j,\nu)-\varepsilon \le \sup_{h\in\mathbb{N}}\left\langle g^{\varepsilon}_h(i)-g^{\varepsilon}_h(j),\, \nu\right\rangle\le f(i,j,\nu)+\varepsilon  \quad \text{for all } (i,j,\nu) \in \R^d\times \R^d\times \mathbb{S}^{d-1} \text{ with } i \neq j. 
\end{equation}
Let us also recall that conservative $C^1$-vector fields are \emph{curl-free}, where the curl is defined by ${\rm curl}(g) =  (\partial_i g_j - \partial_j g_i)_{i,j=1,\ldots,d}$  for $g \in C^1(\R^d;\R^d)$.  
\end{remark}

\begin{remark}\label{rem: take sup}
It follows from the definition that the class of symmetric jointly convex functions is closed under finite sum and countable supremum, provided  the latter is pointwise finite. 
\end{remark}

The main result of this section addresses  the relation of symmetric joint convexity and $BD$-ellipticity, as well as lower semicontinuity of the corresponding  functionals.

\begin{theorem}[Symmetric joint convexity implies $BD$-ellipticity]\label{thm: LSC symjointconvex}
Any symmetric jointly convex function $f\colon\R^d\times \R^d\times \R^d \rightarrow [0,+\infty)$ is $BD$-elliptic,   and the corresponding functional $\mathcal{F}$ defined in (\ref{eq: basic problem}) is lower semicontinuous on $PR(\Omega)$. 
\end{theorem}

\begin{proof}
 We divide the proof into two steps: first, we prove the statement if $f$ is bounded,  continuous,  and satisfies \eqref{eq: continuity}, then we come to the general case.
 
\noindent {\it Step 1.} Assume, in addition, that $f$ is bounded,  continuous  on $\mathbb{R}^d\times \mathbb{R}^d \times \mathbb{S}^{d-1}$ and satisfies \eqref{eq: continuity}.  Fix a triple $(i,j,\nu)\in \R^d\times \R^d \times \mathbb{S}^{d-1}$,  $i\neq j$.  Let $u \in PR(Q^\nu_1)$ be such that $\{u\neq u_{i,j,\nu}\}\subset\!\subset Q^\nu_1$, where $u_{i,j,\nu}$ is the function  defined in (\ref{u_0}). In view of \eqref{eq: PR-def}, we can write $u  =  \sum\nolimits_{k\in \N} a_{Q_k,b_k} \chi_{P_k}$, where $P_1 = \lbrace u=i\rbrace$ and $P_2 = \lbrace u =j\rbrace$.  Fix any $\eps>0$ and define $v = \sum\nolimits_{k=1}^K a_{Q_k,b_k} \chi_{P_k}$ such that  $\mathcal{H}^{d-1}(\bigcup_{k \ge K+1} \partial^* P_k) \le\eps $ and $\{v\neq u_{i,j,\nu}\}\subset\!\subset Q^\nu_1$. Then, we clearly have $v \in BV(Q^\nu_1;\R^d) \cap L^\infty(Q^\nu_1;\R^d)$, and 
\begin{align}\label{eq: eps-bound}
\int_{J_v}f(v^+,v^-,\nu_v)\,{\rm d}\mathcal{H}^{d-1} \le  \int_{J_u}f(u^+,u^-,\nu_u)\,{\rm d}\mathcal{H}^{d-1}  + \eps\Vert f\Vert_\infty. 
\end{align}
Now, it suffices to prove
\begin{align}\label{eq: v proof}
\int_{J_v}f(v^+,v^-,\nu_v)\,{\rm d}\mathcal{H}^{d-1} \geq f(i,j,\nu).
\end{align}
In fact, \eqref{eq: eps-bound}  and the arbitrariness of $\eps$ then show that $f$ is $BD$-elliptic.

To see \eqref{eq: v proof}, we first fix  $g\in C^1(\R^d;\R^d)\cap W^{1,\infty}(\R^d;\R^d)$, and use the chain rule in $BV$ (see \cite[Theorem 3.96]{Ambrosio-Fusco-Pallara:2000}) to obtain $g(v) \in BV(Q^\nu_1;\R^d)$ with 
\begin{align}\label{eq: chain-rule}
Dg(v)= \nabla g(v)\nabla v\mathcal{L}^d+\left(g(v^+)-g(v^-)  \right)\otimes \nu_v\mathcal{H}^{d-1}\mres J_v,
\end{align} 
where $Dg(v) = (D_k g_l(v))_{k,l=1,\ldots,d}$ denotes the  distributional  derivative.  Since $g(v)-g(u_{i,j,\nu})$ has compact support in $Q^\nu_1$, there holds $Dg(v)(Q^\nu_1)=Dg(u_{i,j,\nu})(Q^\nu_1)$. In particular, 
\begin{align}\label{eq: traces}
\mathrm{tr}\big(Dg(v)(Q^\nu_1) \big)=\mathrm{tr}\big(Dg(u_{i,j,\nu})(Q^\nu_1) \big),
\end{align}
 where ``tr'' stands for the trace, i.e., $\mathrm{tr}\left(Dg(v)(Q^\nu_1) \right) = \sum_{k=1}^d D_{k}g_k(v)(Q^\nu_1)$.  Now assume that $g$ is also conservative, i.e., curl-free. We get by \eqref{eq: chain-rule} that 
\begin{align*}
\mathrm{tr}\big(Dg(v)(Q^\nu_1) \big)
&= \int_{Q^\nu_1}\nabla g(v):(\nabla v)^T\, {\rm d}\mathcal{L}^d+ \int_{J_v}\left\langle g(v^+)-g(v^-), \nu_v  \right\rangle\, {\rm d}\mathcal{H}^{d-1}.
\end{align*}
Since $g$ is curl-free and thus $\nabla g(v)$ is a symmetric matrix, whereas $\nabla v$ is a skew symmetric matrix pointwise a.e., we then get
$$\mathrm{tr}\big(Dg(v)(Q^\nu_1) \big) =  \int_{J_v}\left\langle g(v^+)-g(v^-), \nu_v  \right\rangle\, d\mathcal{H}^{d-1}.$$
In a similar fashion, we obtain $\mathrm{tr}\left(Dg(u_{i,j,\nu})(Q^\nu_1) \right) = \left\langle g(i)-g(j), \nu  \right\rangle$. Therefore,  by \eqref{eq: traces} we derive 
\begin{align*}
\int_{J_v}\left\langle g(v^+)-g(v^-), \nu_v  \right\rangle\, {\rm d}\mathcal{H}^{d-1}= \mathrm{tr}\big(Dg(v)(Q^\nu_1) \big)= \mathrm{tr}\big(Dg(u_{i,j,\nu})(Q^\nu_1) \big) = \left\langle g(i)-g(j), \nu  \right\rangle.
\end{align*}
Let $g^{\varepsilon}_h\in C^1(\R^d;\R^d) \cap W^{1,\infty}(\R^d;\R^d)$ be  curl-free for every $h\in\mathbb{N}$ as in Remark  \ref{rem: symm-conv}. Then, taking the supremum   on both sides of the above relation for $g=g^{\varepsilon}_h$ and using \eqref{eq: approximation}  we get 
\begin{align*}
\begin{split}
\int_{J_v}f(v^+,v^-,\nu_v)\,d\mathcal{H}^{d-1}+\varepsilon \mathcal{H}^{d-1}(J_v)  &\geq \sup_{h\in\mathbb{N}}\int_{J_v}\left\langle g^{\varepsilon}_h(v^+)-g^{\varepsilon}_h(v^-), \nu_v  \right\rangle\, d\mathcal{H}^{d-1}\\
&=\sup_{h\in\mathbb{N}}\left\langle g^{\varepsilon}_h(i)-g^{\varepsilon}_h(j), \nu  \right\rangle\geq  f(i,j,\nu)-\varepsilon.
\end{split}
\end{align*}
By the arbitrariness of $\varepsilon >0$, this shows \eqref{eq: v proof}. By using also Theorem \ref{thm: LSC iff BD-ell + bounded} we get that  $\mathcal{F}$ defined in \eqref{eq: basic problem} is lower semicontinuous in $PR(\Omega)$. This  concludes the proof of Step 1.

\noindent {\it Step 2.}  We now address the general case.  For each $M \in \mathbb{N}$, set $\hat g_h= g_h$ if $h \le M$, and $\hat g_h=0$  else.  Consider $ f_M(i, j, \nu):= \sup_{h\in\mathbb{N}}\left\langle \hat g_h(i)-\hat g_h(j),\,  \nu\right\rangle$ (so that in particular, $f_M \ge 0$). Clearly, Step $1$ can be applied to the functions $f_M$ as each function $g_h$ in \eqref{symmetric joint convexity}  is bounded and   uniformly continuous. Since $f=\sup_M f_M$, the conclusion follows from Corollary \ref{cor: f=sup fh is BD and LSC}.  
\end{proof}


%
%
%
%
%
%
%
%
%
%

We close this section by providing a prototypical class of symmetric jointly convex functions.

\begin{example}[Prototype of symmetric jointly convex functions]\label{ex: prototyp}
Given any orthonormal basis  $\{ \xi_1,\ldots,\xi_d \}$  of $\R^d$ and bounded, uniformly continuous functions $h_k \in C(\R)$, $k=1,\ldots,d$, consider the function 
$$
g(w):= \sum_{k=1}^d h_k \big(\langle w, \xi_k\rangle\big) \, \xi_k \ \ \ \ \text{for all } w \in \R^d.
$$
Then, clearly $g\in C(\R^d;\R^d)$ is bounded, uniformly continuous, and conservative with potential
$$G(w) = \sum_{k=1}^d H_k \big(\langle w, \xi_k\rangle\big)\ \ \ \ \text{for all } w \in \R^d,$$
where $H_k$ denotes a primitive of $h_k$. Then, functions $f$ as in \eqref{symmetric joint convexity} with functions $g_h$ of the above form are symmetric jointly convex. We will exploit this several times in the examples in Section \ref{sec: examples}.

\end{example}

\section{Examples of $BD$-elliptic functionals}\label{sec: examples}

In this section, we present various examples of functions that are $BD$-elliptic. We start with some classes of symmetric jointly convex functions, including in particular the density $(i,j,\nu) \mapsto |i-j| |\nu|$. Afterwards, we consider so-called \emph{biconvex functions}, and then functions which either only depend on the normal or have a ``mild'' dependence on  the traces  $i$ and $j$. Finally, we provide examples of functions that are $BV$-elliptic but  not $BD$-elliptic.

\subsection{Subadditive isotropic integrands}\label{sec: joint2}

In this subsection, we show that certain subadditive isotropic integrands are $BD$-elliptic. This result constitutes one of our main results since the  class of considered  functions contains  in particular   the density $(i,j,\nu) \mapsto |i-j| |\nu|$.

\begin{theorem}[Subadditive, isotropic integrands]\label{example: g subadditive}
Let $g:[0,+\infty)\rightarrow [0,+\infty)$ be an increasing function satisfying
\begin{equation}\label{eq: mild subadditivity}
\frac{g(t)}{t} \quad \mbox{is nonincreasing on }(0, +\infty)\,.
\end{equation}
Then, $f\colon\R^d\times\R^d\times\R^d\rightarrow [0,+\infty)$ defined as
\begin{align}\label{eq: g,nu}
f(i,j,\nu):= g(|i-j|)|\nu|
\end{align}
is symmetric jointly convex and thus $BD$-elliptic. In particular, the function 
\begin{align*}
(i,j,\nu)\mapsto|i-j||\nu|
\end{align*}
is symmetric jointly convex.
\end{theorem}

We remark  that \eqref{eq: mild subadditivity}  particularly  implies that $g$ is \emph{subadditive}.  This condition   is satisfied, for instance, if $g$ is concave (as we have $g(0)\ge 0$).   By choosing $g \equiv 1$, we re-derive the well-known fact that the Hausdorff-measure $\mathcal{H}^{d-1}$ is lower semicontinuous on $PR(\Omega)$,  see \cite[Theorem 11.3]{DM} or \cite[Lemma 3.3]{FM}. 

The notion of \emph{isotropy} refers to the fact that $f(i,j,\nu) = f(i,j,R\nu)$  and $f(i,j,\nu) = f(Ri,Rj,\nu)$ for all  proper rotations  $R \in SO(d)$. Note that this class is much smaller than the corresponding class of $BV$-elliptic functions considered in \eqref{eq: BV class} where $f$ can be anisotropic as long as $\theta$ is a pseudo-distance and  $\psi$  is  even, positively $1$-homogeneous, and convex. In fact, as we will show in Subsection~\ref{sec: counterexamples} below, for certain anisotropies it turns out that the functions in \eqref{eq: BV class} are $BV$-elliptic, but not $BD$-elliptic.

The proof of Theorem \ref{example: g subadditive} follows directly from the following lemma.

\begin{lemma}\label{lemma: special functions}
For each $M,a  \ge  0$ the function  $\theta_{M,a}\colon  \R^d \times \R^d \times \R^d   \rightarrow \R$ defined by
\begin{align*}
\theta_{M,a}(i,j,\nu):= \min\{a|i-j|,M\}\, |\nu| \quad\quad \text{ for all } (i,j,\nu) \in \R^d \times \R^d \times \R^d
\end{align*}
is  symmetric jointly convex.
\end{lemma}

\begin{proof}[Proof of Theorem \ref{example: g subadditive}]
Let $f$ as in \eqref{eq: g,nu} be given. For each $t>0$, $t \in \mathbb{Q}$, we choose $M_t  \ge  0$ and $a_t  \ge  0$ such that $M_t = g(t)$ and $a_tt = g(t)$. Then the  monotonicity of $g$ and \eqref{eq: mild subadditivity}  imply $\min\{a_t z ,M_t\} \le g(z)$ for all $z >0$ and $\min\{a_t t,M_t\}= g(t)$. This yields
$$
f(i,j,v) = \sup_{t >0, t \in \mathbb{Q}} \theta_{M_t,a_t}(i,j,\nu) \quad \quad \text{for all } (i,j,\nu) \in \R^d\times \R^d \times \R^d\quad \text{with } i \neq j.
$$
Consequently, as each $\theta_{M_t,a_t}$ is symmetric jointly convex, also $f$ is symmetric jointly convex,  see Remark \ref{rem: take sup}.   
\end{proof}

The remainder of this subsection is devoted to the proof of Lemma \ref{lemma: special functions}.  Let us start with a technical lemma that each element of $\mathbb{S}^{d-1}$ can be mapped to any other element of $\mathbb{S}^{d-1}$ by a symmetric matrix. In the following,   $\Vert\cdot\Vert$  denotes the operator norm of a matrix $B \in \R^{d \times d}$, i.e., $\Vert B \Vert := \max_{\nu \in \mathbb{S}^{d-1}}|B\nu|$. 
\begin{lemma}\label{lem: Bu=v}
For all $u,v\in \mathbb{S}^{d-1}$, there exists a symmetric matrix $B\in \mathbb{R}^{d\times d}_{\rm sym}$ such that  $\Vert B \Vert=1$ and $Bu=v$.
\end{lemma}

\begin{proof}
First, observe that the result is trivial if $u=\pm v$ by choosing $B = \pm {\rm Id}$. We start by proving the statement for $d=2$ (Step 1) and then address the general case (Step 2). 

\noindent \emph{Step 1.} We prove the statement for $d=2$. Let $u=(\cos\alpha,\sin\alpha), v=(\cos\beta,\sin\beta) \in \mathbb{S}^1$ with $\alpha,\beta\in [0,2\pi)$. We let  $\gamma=\alpha+\beta$ and introduce the following matrix which is a composition of a rotation and a reflection: 
$$
B:= \left( \begin{matrix} \cos\gamma & -\sin\gamma \\ \sin\gamma & \cos\gamma \end{matrix} \right)\left(\begin{matrix} 1 &  0 \\ 0 & -1 \end{matrix} \right) = \left( \begin{matrix} \cos\gamma & \sin\gamma \\ \sin\gamma & -\cos\gamma \end{matrix} \right).
$$
Note that $B$ is symmetric. By using the angle sum identities $\cos(\gamma-\alpha) = \cos\alpha\cos\gamma + \sin\alpha\sin\gamma $  and $\sin(\gamma-\alpha) = \cos\alpha\sin\gamma - \sin\alpha\cos\gamma$ we get $Bu=v$. Moreover, as $|B w| = |w|$ for all $w \in \R^2$, we find that $\Vert B \Vert=1$.  This concludes the first step.

\noindent \emph{Step 2.} Let $u,v\in\mathbb{S}^{d-1}$ with  $u\neq\pm v$. Let us consider the two-dimensional plane $\Pi_{u,v}$ in $\R^d$ spanned  by the two vectors $u$ and  $v$. Fix an orthonormal basis $\xi_1, \xi_2 \in \R^d$ of $\Pi_{u,v}$  and  note that $u = \cos(\alpha) \xi_1 + \sin(\alpha) \xi_2$ and  $v = \cos(\beta) \xi_1 + \sin(\beta) \xi_2$ for some $\alpha,\beta \in [0,2\pi)$. We define the matrix 
\begin{align*}
B = \cos(\gamma)  \big(\xi_1 \otimes \xi_1 - \xi_2 \otimes \xi_2   \big) + \sin(\gamma) \big(\xi_1 \otimes \xi_2 + \xi_2 \otimes \xi_1  \big)  ,
\end{align*}
where $\gamma = \alpha + \beta$. As in Step 1, we can check that $B \in \R^{d \times d}_{\rm sym}$, $\Vert B \Vert=1$, and $Bu=v$.   This concludes the proof.
\end{proof}

In the following, the symmetry of the matrices given in Lemma \ref{lem: Bu=v} will be crucial as it allows us to diagonalize the matrices  and to represent the function $f$ in \eqref{eq: g,nu} in terms of functions similar to the prototype introduced in Example \ref{ex: prototyp}.  We are now in a position to prove Lemma \ref{lemma: special functions}.

\begin{proof}[Proof of Lemma \ref{lemma: special functions}]

 Without restriction we assume that $M,a>0$.  We start by defining the functions $g_h$ (Step 1) and then show equality in \eqref{symmetric joint convexity} (Steps 2--3).

\noindent \emph{Step 1: Definition of the functions $g_h$.} We start by introducing the class of bounded, uniformly continuous, and conservative vector fields $(g_h)_{h \in \N} \in C(\R^d;\R^d)$. Given  $M, a >0$,  define $\eta_M\colon\R\rightarrow [0,+\infty)$ by $\eta_M(t):=   \min\{|t|,M\}$ for $t\in \R$, i.e., 
\begin{align}\label{eq: basic identity}
\theta_{M,a}(i,j,\nu)= \eta_M(a|i-j|)|\nu| \quad \quad \text{for $(i,j,\nu) \in \R^d \times \R^d \times \R^d$}.
\end{align}
It is elementary to check that $\eta_M$ is an even, uniformly continuous, and  subadditive function. 
Consider a symmetric matrix $B$ with $\Vert B \Vert=1$, as provided  by Lemma \ref{lem: Bu=v}.  Let $\{ \xi_1,\ldots,\xi_d \}$ be an orthonormal basis of $\R^d $ made of eigenvectors of $B$ and let $\{\lambda_1,\ldots,\lambda_d \}$ be the set of corresponding eigenvalues. Note that $|\lambda_k| \le 1$ for $k=1,\ldots,d$ since $\Vert B \Vert = 1$.   Finally, let us fix $\mu\in\mathbb{S}^{d-1}$ and $c\in \R^d$.   We define $g_{B,\mu,c}\colon \R^d \to \R^d$ by
\begin{align}\label{eq: g-def}
g_{B,\mu,c}(w):= \sum_{k=1}^d\lambda_k \mu_k \,\eta_M\left(a\left(\frac{\langle w, \xi_k\rangle}{\mu_k}-c_k  \right)\right)\xi_k \ \ \ \ \text{for all } w \in \R^d.
\end{align}
 Here and in the following, an addend is interpreted to be zero whenever $\mu_k=0$. Clearly, each $g_{B,\mu,c}$ is bounded and  uniformly continuous. It is elementary to check that $g_{B,\mu,c}$ is conservative with potential
$$\sum_{k=1}^d \frac{\lambda_k \mu_k^2}{a} \, \Theta_M\left(a\left(\frac{\langle w, \xi_k\rangle}{\mu_k}-c_k  \right)\right), $$ 
where $\Theta_M$ denotes a primitive of $\eta_M$, cf.\ the prototypes discussed in Example \ref{ex: prototyp}.  

We denote by $(u_k,v_k)_k$ a countable dense set in $\mathbb{S}^{d-1}\times \mathbb{S}^{d-1}$ and denote by $(B^k)_k$ the symmetric  matrices from     Lemma \ref{lem: Bu=v}  satisfying $B^k\,u_k=v_k$. Moreover, let $(\mu^l)_l$ be a countable, dense set in $\mathbb{S}^{d-1}$ and let $(c^n)_n$ be a countable, dense set  in $\R^d$.   To shorten the notation, we label the countable set of functions $(g_{B^k,\mu^l,c^n})_{k,l,n}$ by $(g_h)_{h \in \N}$.

Let us now show that $\theta_{M,a}$ is symmetric jointly convex, namely, for every $(i,j,\nu) \in \R^d \times \R^d \times \R^d$,  $i\neq j$,  there holds  
\begin{align}\label{prop: |.|_M joint sym convex}
\theta_{M,a}(i,j,\nu) = \sup_{h\in\mathbb{N}}\left\langle g_h(i)-g_h(j),\nu   \right\rangle.
\end{align}
We split the proof into two inequalities. Before we enter into the details, let us briefly explain the rough ideas behind the parameters $B$, $\mu$, and $c$: we will choose $B$, $\mu$, $c$ in an optimal way in order to obtain one inequality, see \eqref{eq: B choice}--\eqref{eq: mu choice}. In particular, we can choose $B$ such that $g_{B,\mu,c}(i)-g_{B,\mu,c}(j)$ and $\nu$ are aligned. Moreover, $c$ can be selected such that $g_{B,\mu,c}(i)-g_{B,\mu,c}(j) = g_{B,\mu,0}(i-j)$. (This is inspired by \cite[Example 5.23]{Ambrosio-Fusco-Pallara:2000}.) Finally, $\mu$ will allow us to deal with the nonlinearity of $\eta_M$.

\noindent \emph{Step 2: Proof of ``$\le$''.}  We consider the symmetric matrix $B$ given by Lemma \ref{lem: Bu=v} such that 
\begin{align}\label{eq: B choice}
\frac{B(i-j)}{|i-j|}=\frac{\nu}{|\nu|}.
\end{align}
Moreover, we choose $\mu=(\mu_1,\mu_2,\dots,\mu_d)\in\mathbb{S}^{d-1}$ by 
\begin{align}\label{eq: do not remove}
\mu_k=\frac{\langle i-j,\xi_k \rangle}{|i-j|}  \ \text{ for $k=1,\ldots,d$}.
\end{align}
By  $c=(c_1,c_2,\dots,c_d)\in\R^d$ we denote  the vector with  
\begin{align}\label{eq: mu choice}
c_k= \langle j,  \xi_k \rangle /\mu_k,
\end{align}
whenever $\mu_k\neq 0$ and $c_k = 0$ else.  For brevity, we write $g=g_{B,\mu,c}$. In view of \eqref{eq: g-def} and  by the choices of $\mu$ and $c$, we get 
\begin{align*}
\left\langle g(i)-g(j),\nu  \right\rangle &= \left\langle \sum_{k=1}^d \lambda_k \mu_k \left[\eta_M\left( a\frac{\langle i, \xi_k\rangle}{\mu_k} - ac_k  \right)-\eta_M\left( a\frac{\langle j, \xi_k \rangle}{\mu_k} - ac_k  \right)  \right]\xi_k,\nu  \right\rangle \\&= \left\langle \sum_{k=1}^d \lambda_k \mu_k \,  \eta_M\left( a\frac{ \langle i-j, \xi_k \rangle}{\mu_k}  \right) \xi_k,\nu  \right\rangle= \left\langle \sum_{k=1}^d \lambda_k \frac{\langle i-j, \xi_k \rangle}{|i-j|}  \eta_M\left( a|i-j|  \right) \xi_k,\nu  \right\rangle.
\end{align*}
Since $(\xi_k)_k$ is an orthonormal basis of $\R^d $ made of eigenvectors of $B$ and  $(\lambda_k)_k$ are the corresponding eigenvalues, we get by \eqref{eq: basic identity} and \eqref{eq: B choice}  
\begin{align*}
\left\langle g(i)-g(j),\nu  \right\rangle &= \frac{\eta_M\left( a|i-j|  \right)}{|i-j|}\left\langle\sum_{k=1}^d\lambda_k \langle i-j, \xi_k\rangle\,\xi_k,\nu \right\rangle = \frac{\eta_M\left( a|i-j|  \right)}{|i-j|}  \left\langle B(i-j),\nu \right\rangle \\& = \eta_M( a|i-j| )|\nu|   =  \theta_{M,a}(i,j,\nu).
\end{align*}
 By the density of $(u_k,v_k)_k$, $(\mu^l)_l$, and  $(c^n)_n$ we get that the function $g$ considered above can be approximated by $(g_h)_{h \in \N}$.   Thus,  we obtain inequality  ``$\le$'' in \eqref{prop: |.|_M joint sym convex}. 

\noindent \emph{Step 3: Proof of ``$\ge$''.} Fix any $g=g_{B,\mu,c}$ as above.  For brevity, we define 
$$b_k:= \eta_M\left( a\frac{\langle i, \xi_k\rangle}{\mu_k} - ac_k  \right)-\eta_M\left( a\frac{\langle j, \xi_k\rangle}{\mu_k} - ac_k  \right)$$
for $k=1,\ldots,d$. Since $\eta_M$ is nonnegative, subadditive, and even, we get 
\begin{align}\label{eq: b bound}
|b_k| \le  \eta_M\left( a\frac{\langle i-j, \xi_k\rangle}{\mu_k}\right).
\end{align}
Since  $|\lambda_k| \le 1$ for $k=1,\ldots,d$ (recall $\Vert B \Vert =1$) and $( \xi_k)_k$ forms an orthonormal basis, we get for every $h\in\mathbb{N}$  by  \eqref{eq: g-def} and  the Cauchy-Schwarz inequality 
\begin{align*}
\left|\left\langle g(i)-g(j),\nu  \right\rangle\right|&= \left| \left\langle \sum_{k=1}^d \lambda_k \mu_k b_k\xi_k,\nu  \right\rangle\right| \le   \left( \sum_{k=1}^d(\mu_kb_k)^2\right)^{1/2}|\nu|.
\end{align*}
We now distinguish two cases: if $a|i-j| \le M$, we deduce from  \eqref{eq: basic identity}, \eqref{eq: b bound}, and the fact that $\eta_M(t) \le |t|$ for $t \in \R$ that
\begin{align*}
\left|\left\langle g(i)-g(j),\nu  \right\rangle\right|&\le   \left( \sum_{k=1}^d\mu_k^2 \left( a\frac{\langle i-j, \xi_k\rangle}{\mu_k}\right)^2 \right)^{1/2}|\nu| = a|i-j||\nu| = \theta_{M,a}(i,j,\nu).
\end{align*}
 Otherwise, if $a|i-j| > M$, in view of  \eqref{eq: basic identity}, \eqref{eq: b bound}, we find by using $\Vert \eta_M \Vert_\infty \le M$ and  $\mu \in \mathbb{S}^{d-1}$  that
 \begin{align*}
\left|\left\langle g(i)-g(j),\nu  \right\rangle\right|&\le   \left( \sum_{k=1}^d M^2\mu_k^2 \right)^{1/2}|\nu| = M|\nu| = \theta_{M,a}(i,j,\nu).
\end{align*} 
Taking the supremum over all $(g_h)_{h \in \N}$ we obtain inequality ``$\ge$'' in \eqref{prop: |.|_M joint sym convex}. This concludes the proof. 
\end{proof}

\BBB

\begin{example}
An example of symmetric jointly convex functions is given by the class of functions considered in \cite{GZ1}, namely 
$$f(i,j,\nu) =  \varphi\big( |\langle i-j, \nu \rangle| \big)  $$
for   convex, subadditive, and increasing functions $\varphi \colon [0,\infty) \to [0,\infty)$. In fact, by \cite[Theorem 1.2]{GZ1} such functions can be written as the supremum of functions of the form $\bar{f}(i,j,\nu) = a + b |\langle i-j, \nu \rangle|$ for $a,b \ge 0$. Therefore, in view of Remark \ref{rem: take sup}, we need to check that $\bar{f}$ is symmetric jointly convex for given $a,b \ge 0$. The constant function $a$ has this property by Theorem \ref{example: g subadditive}. Moreover, the symmetric joint convexity of $(i,j,\nu) \mapsto b|\langle i-j, \nu \rangle|$ follows from Definition \ref{def: symm-conv} with $g_1(x) = bx$ and $g_2(x) = -bx$ for $x \in \R^d$.    
\end{example}
\EEE

\subsection{A further class of symmetric jointly functions}\label{sec: orlando}

In this subsection, we revisit a class of functions considered  in a more general context  in  \cite{DalMOT}, where the authors prove that the associated energy functionals, see \eqref{eq: basic problem}, are lower semicontinuous.  Given even,  continuous,  and subadditive function $\theta_k\in C(\R;[0,+\infty))$,  $k=1,\ldots,d$,  with $\theta_k(0) = 0 $,  we define the function $f\colon\R^d\times\R^d\times \R^d \rightarrow [0,+\infty)$ by
\begin{align}\label{eq: G DMasoRO}
f(i,j,\nu):= \sup_{(\xi_1,\dots,\xi_k)}\left(\sum_{k=1}^d \theta_k \big( \langle i - j  ,  \xi_k \rangle\big)^2 \, | \langle \nu, \xi_k\rangle|^2 \right)^{1/2},
\end{align}
where the supremum is taken over all orthonormal bases $(\xi_k)_{k=1}^d$ of $\R^d$.

We prove that functions of this form are symmetric jointly convex which provides an alternative (and in our opinion simpler)  approach to the lower semicontinuity of the  functional in \eqref{eq: basic problem} for $f$ as above, when restricted to $PR(\Omega)$.   Let us also mention that, in contrast to the \BBB class considered in \eqref{eq: g,nu}, \EEE the functions in \eqref{eq: G DMasoRO}  may in general be anisotropic.

\begin{proposition}
Let $\theta_k\in C(\R;[0,+\infty))$ be even,   continuous,  and subadditive functions with $\theta_k(0) = 0$  for $k=1\ldots,d$.  Then, the function $f\colon\R^d\times\R^d\times  \R^d \rightarrow [0,+\infty)$ defined in (\ref{eq: G DMasoRO}) is symmetric jointly convex.
\end{proposition}
\begin{proof}
 We start by noticing that each $\theta_k$ is uniformly continuous. In fact, suppose by contradiction that this was false. Then, there exists $\delta>0$ and a sequence of pairs $(x_n,y_n) \in \R^2$ such that $|x_n -y_n| \to 0$ as $n \to \infty$ and $\theta_k(x_n) \le \theta_k(y_n) - \delta$ for all $n \in \N$. But as  $\theta_k$ is subadditive and even, we get for $n$ large enough that  $\theta_k(y_n) \le \theta_k(x_n) + \theta(|x_n-y_n|) \le \theta_k(x_n) + \delta/2$, where the last step follows from  $|x_n -y_n| \to 0$ and $\theta_k(0)=0$. This is a contradiction. 

We also observe that it is not restrictive to assume that each $\theta_k$ is bounded. In fact, otherwise we consider the truncations $\theta_k^M$ defined by $\theta_k^M(t) := \min \lbrace \theta_k(t), M\rbrace$ for $t \in \R$  which are again even, uniformly continuous, and subadditive. Then, by $\sup_{M>0} \theta_k^M = \theta_k$ and Remark \ref{rem: take sup} it clearly suffices to prove that $f$ in \eqref{eq: G DMasoRO} with $\theta_k^M$ in place of $\theta_k$ is symmetric jointly convex. For simplicity, we assume in the following that each $\theta_k$ is bounded. 

  By definition of $f$  and Remark \ref{rem: take sup},  it is sufficient to show that the function $\bar{f}\colon\R^d\times\R^d\times \R^d\rightarrow [0,+\infty)$ defined by 
$$
\bar{f}(i,j,\nu):= \left(\sum_{k=1}^d \theta_k\big( \langle i-j,  \xi_k\rangle\big)^2\, |\langle \nu, \xi_k\rangle|^2 \right)^{1/2}
$$
is symmetric jointly convex, where $(\xi_k)_k$ is any orthonormal basis of $\R^d$. To this end, for each $p\in \mathbb{Q}^d$ with $|p|\leq 1$, each $q\in\mathbb{Q}^d$, and each $\sigma\in \{-1,1 \}^d$, we define the conservative vector field $g_{p,q,\sigma}\colon\R^d \rightarrow \R^d$ by
\begin{align*}
g_{p,q,\sigma} (w):=  \sum_{k=1}^d  \sigma_k\langle p, \xi_k \rangle\, \theta_k\big( \langle w-q, \xi_k\rangle\big)\,\xi_k \ \ \ \ \text{for all }  w \in \R^d, 
\end{align*} 
cf.\ Example \ref{ex: prototyp}. Clearly,  $g_{p,q,\sigma}$ is bounded and uniformly continuous.  Our goal is to prove that
\begin{align}\label{eq: goal DalMaso Toader Orlando}
\bar{f}(i,j,\nu)= \sup_{p,q,\sigma}\left\langle g_{p,q,\sigma}(i)-g_{p,q,\sigma}(j),\nu \right\rangle  \quad \quad \text{for all } (i,j,\nu) \in \R^d\times \R^d\times \R^d,  \ i \neq j. 
\end{align}
We show the two inequalities separately. 

\noindent \emph{Step 1: Proof of ``$\ge$''.} Fix $(i,j,\nu)\in\R^d\times \R^d\times \R^d$  with $i\neq j$.  Consider $p\in \mathbb{Q}^d$ with $|p|\leq 1$, $q\in\mathbb{Q}^d$, and $\sigma\in \{-1,1 \}^d$. We get
\begin{align}\label{pinco}
\big| \left\langle g_{p,q,\sigma}(i)-g_{p,q,\sigma}(j), \nu \right\rangle  \big| \leq \sum_{k=1}^d\left|  \theta_k\big( \langle i-q, \xi_k\rangle \big)-\theta_k\big( \langle j-q, \xi_k\rangle\big)  \right||\langle p, \xi_k\rangle| |\langle \nu, \xi_k\rangle|.
\end{align}
Since $\theta_k$ is nonnegative, subadditive, and even, we obtain 
\begin{align*}
\big|  \theta_k\big( \langle i-q, \xi_k\rangle\big)- \theta_k\big(\langle j-q, \xi_k\rangle\big)\big|\leq \theta_k\big( \langle i-j, \xi_k \rangle\big)\quad  \quad \text{ for all }  k=1,\ldots,d.
\end{align*}
Then, by (\ref{pinco}), the Cauchy-Schwarz inequality, and the fact that $(\xi_k)_k$ is an orthonormal basis  we get
\begin{align*}
\big| \left\langle g_{p,q,\sigma}(i)-g_{p,q,\sigma}(j),\nu \right\rangle  \big| &\leq \sum_{k=1}^d \theta_k\big( \langle i-j, \xi_k\rangle\big)  |\langle \nu, \xi_k\rangle||\langle p,  \xi_k\rangle|\leq \left( \sum_{k=1}^d \theta_k\big(\langle i-j, \xi_k\rangle\big)^2 |\langle \nu, \xi_k\rangle|^2   \right)^{1/2}\hspace{-0.2cm}|p|.
\end{align*}
By recalling $|p|\le1$ and by passing to  the supremum over $(p,q,\sigma)$ in their domain, we obtain inequality  ``$\ge$'' in \eqref{eq: goal DalMaso Toader Orlando}. 

\noindent \emph{Step 2: Proof of ``$\le$''.} Fix $(i,j,\nu)\in\R^d\times \R^d\times \R^d$  with $i \neq j$.  Let us set $\sigma_k= \mathrm{sign} \langle \nu, \xi_k\rangle$ for $k=1,\dots,d$. Moreover, consider a sequence of vectors $(q_h)_{h\in\mathbb{N}} \subset \mathbb{Q}^d$ such that $\lim_{h\rightarrow \infty}q_h= j$. Then, for each $p \in \mathbb{Q}^d$ with $|p|\leq 1$, we get by $\theta_k(0) = 0 $ that 
\begin{align}\label{panco}
\lim_{h\rightarrow \infty}\left\langle g_{p,q_h,\sigma}(i)-g_{p,q_h,\sigma}(j), \nu \right\rangle = \sum_{k=1}^d \langle p, \xi_k \rangle \, \theta_k\big(\langle i-j, \xi_k\rangle\big) |\langle \nu, \xi_k\rangle|= \langle p,  \mu\rangle,
\end{align}
where the vector  $\mu \in \R^d$ is defined by
\begin{align*}
\mu:= \sum_{k=1}^d \theta_k\big( \langle i-j, \xi_k\rangle \big) |\langle \nu, \xi_k\rangle| \xi_k.
\end{align*}
We choose $(p_m)_{m\in\mathbb{N}}\subset \mathbb{Q}^d$, $|p_m|\le 1$, such that $\lim_{m\rightarrow \infty}p_m=\mu/|\mu|$. Then, by (\ref{panco}) we conclude  
\begin{align*}
\sup_{p,q,\sigma}\left\langle g_{p,q,\sigma}(i)-g_{p,q,\sigma}(j),\nu \right\rangle &\ge  \lim_{m\rightarrow \infty}\lim_{h\rightarrow \infty}\left\langle g_{p_m,q_h,\sigma}(i)-g_{p_m,q_h,\sigma}(j),\nu \right\rangle = |\mu|\\& = \left(\sum_{k=1}^d \theta_k\big( \langle i-j,  \xi_k\rangle\big)^2|\langle \nu, \xi_k\rangle|^2 \right)^{1/2} = \bar{f}(i,j,\nu).
\end{align*}
This proves  inequality  ``$\le$'' in \eqref{eq: goal DalMaso Toader Orlando}. 
\end{proof}

\BBB

\begin{example}
In \cite{GZ2}, functions of the form 
$$f(i,j,\nu) := \sup_{\xi \in \mathbb{S}^{d-1}}   |\langle \nu, \xi \rangle  | \,  \psi\big( |\langle i-j, \xi \rangle| \big)  $$
are considered, where $\psi \colon [0,\infty) \to [0,\infty)$ is a continuous, nondecreasing, and  subadditive function. This is a special case of \eqref{eq: G DMasoRO} for $\theta_1 = \psi(|\cdot|)$, $\theta_k = 0$ for $k=2,\ldots,d$, and $\xi = \xi_1$. 
\end{example}

  \EEE

\begin{example}\label{rem: examples DMOT}
Among the integrands of the form \eqref{eq: G DMasoRO}, we may mention $f(i, j, \nu)=|(i-j)\odot \nu|$, where the symbol $|\cdot|$ denotes the Frobenius norm, and  $\odot$ denotes the symmetric tensor product $a \odot b = \frac{1}{2}(a \otimes b + b \otimes a)$ for $a,b\in \R^d$.   This is obtained for $\theta_k(t)=|t|$,  $k=1,\ldots,d$,  in \eqref{eq: G DMasoRO}. If one instead chooses $\theta_k  (t)= \min  \{|t|, M\}$ for a fixed $M>0$, the resulting $f$ is a bounded integrand satisfying $f(i, j, \nu)=|(i-j)\odot \nu|$ when $|i-j|\le M$,  $f(i,j, \nu)=M|\nu|$ when $|i-j|\ge \sqrt2 M$, with a smooth transition in  the annulus between the radii $M$ and $\sqrt2 M$ (see \cite[Section 6]{DalMOT}).
\end{example}

\subsection{Symmetric biconvex functions}\label{sec: biconvex}

In this subsection, we introduce and study \emph{symmetric biconvex functions}.

\begin{definition}[Symmetric biconvexity]
We say that $f\colon\R^d\times \R^d\times  \R^d \to [0,+\infty)$ is a \emph{symmetric biconvex function} if there exists a convex and positively $1$-homogeneous function $\theta\colon\R^{d\times d}_{\rm sym}\rightarrow [0,+\infty)$ such that
$$
f(i,j,\nu)=\theta\big((i-j)\odot \nu\big)\quad \quad \text{for all } (i,j,\nu)\in\R^d\times \R^d\times    \R^d. 
$$ 
\end{definition}
Here, $\odot$ denotes the symmetric tensor product $a \odot b = \frac{1}{2}(a \otimes b + b \otimes a)$ for $a,b\in \R^d$. This notion is related to biconvexity defined in \cite[Section 2.2]{AmbrosioBraides2}.
 For $\theta$ being the Frobenius norm, the corresponding integrand  is symmetric jointly convex (and hence $BD$-elliptic), as discussed in Example \ref{rem: examples DMOT}. In particular, the functional
\[
u\mapsto \int_{J_u} |(u^+-u^-)\odot \nu_u|\,\mathrm{d}\mathcal{H}^{d-1}
\] 
is lower semicontinuous  on $PR(\Omega)$,  see Theorem \ref{thm: LSC symjointconvex}.  In the general case, the situation is more complicated.  We may indeed prove that
\begin{itemize}
\item symmetric biconvex functions are symmetric jointly convex when restricted to compact subsets, in a sense made precise by Proposition \ref{prop: bijoint} below.
\item symmetric biconvex functions with  $\lbrace \theta =0 \rbrace = \lbrace 0 \rbrace$  are $BD$-elliptic, cf.\ Proposition \ref{prop: biconvex}. 
\end{itemize}
\BBB These two \EEE results only allow us, in general, to deduce that for biconvex functions the functional  $\mathcal{F}$ defined in (\ref{eq: basic problem}) is lower semicontinuous in $PR(\Omega)$ along  uniformly bounded sequences, see Corollary \ref{cor: unbounded}.
This can also be inferred by the results  in \cite{BCD}, where lower semicontinuity  in the space $SBD$ of integral functionals corresponding to symmetric biconvex functions has already been addressed. \BBB We emphasize that, in the case  $\lbrace \theta =0 \rbrace = \lbrace 0 \rbrace$, the necessity of the $L^\infty$-bound is only a technical  issue due to our method. Indeed,  lower semicontinuity also holds without this assumption by directly proving a lower semicontinuity result for symmetric jointly convex functions.  We defer the proof of this fact to the next section (see Theorem \ref{thm: biconv}). \EEE

%
%

We now state and prove the announced results. We first address the relation of symmetric biconvex functions to  symmetric jointly convex functions.

\begin{proposition}[Biconvexity and joint convexity]\label{prop: bijoint}
Let $f\colon\R^d\times \R^d\times \R^d\rightarrow [0,+\infty)$ be a symmetric biconvex function. Then, there exists a sequence  $(f_M)_{M \in \N}$  of symmetric jointly convex functions such that $f(i,j,\nu)=f_M(i,j, \nu)$ for all $(i,j, \nu) \in B_M(0)\times B_M(0) \times \R^d$. 
\end{proposition}

\begin{proof}
By assumption, we have  
\begin{align*}
f(i,j,\nu)=\theta\big((i-j)\odot\nu\big),
\end{align*}
with $\theta\colon\R^{d\times d}_{\rm sym} \rightarrow [0,+\infty)$ convex and  positively $1$-homogeneous. It is a well known fact that 
\begin{align*}
\theta (F)= \sup_{Z\in W_{\theta}} F:Z \quad \quad \text{for all } F\in \R^{d\times d}_{\rm sym},
\end{align*}
where $W_{\theta}\subset \R^{d\times d}$ is a bounded set depending on $\theta$.  (Here, the symbol $:$ denotes the scalar product for matrices in $\R^{d\times d}$.)  Let us consider a countable, dense set of matrices $(Z_h)_{h\in\mathbb{N}}$ in $W_{\theta}$. Since $F:Z^{\rm skew}=0$ whenever $F \in \R^{d \times d}_{\rm sym}$ and $Z^{\rm skew}\in \R^{d\times d}_{\rm skew}$, we get
\begin{align}\label{eq: bijoint}
\theta\big((i-j)\odot\nu\big)&= \sup_{h\in\mathbb{N}} \, ((i-j)\odot\nu):Z_h= \sup_{h\in\mathbb{N}} \, ((i-j)\odot\nu) : Z_h^{\rm sym}=
\sup_{h\in\mathbb{N}}\left\langle Z_h^{\rm sym}\, i -Z_h^{\rm sym}\, j,\nu  \right\rangle,
\end{align}
where $Z_h^{\rm sym} := \frac{1}{2}(Z_h^T + Z_h)$. Consequently,
\begin{align}\label{eq: KKK }
\theta\big((i-j)\odot\nu\big)  = \sup_{h\in\mathbb{N}} \,\langle g_h(i) - g_h(j), \nu \rangle, 
\end{align}
where   $g_h$ is defined by $g_h(x) := Z_h^{\rm sym} x$ for $x \in \R^d$.  We define a truncation of each $g_h$ as follows. For $M>0$, we consider the function $\tau_M\colon \R \to \R$ given by $\tau_M(t) = t$ for $|t|\le M$ and $\tau_M(t) = {\rm sgn}(t) M$ else.  
 Let $\{ \xi_1,\ldots,\xi_d \}$ be an orthonormal basis of $\R^d $ made of eigenvectors of $Z_h^{\rm sym}$ and let $\{\lambda_1,\ldots,\lambda_d \}$ be the set of corresponding eigenvalues. Then, we introduce
 $$g^M_h(w) := \sum_{k=1}^d \lambda_k \tau_M\big( \langle w,\xi_k \rangle  \big) \xi_k  \ \ \ \ \text{ for all $w \in \R^d$.} $$
We observe  that  each $g_h^M$ is bounded, uniformly continuous, and conservative, cf.\ the prototype in Example \ref{ex: prototyp}.   Then,  we define the symmetric jointly convex function
\[
f_M(i, j, \nu):=\sup_{h\in\mathbb{N}} \, \langle g^M_h(i) - g^M_h(j), \nu \rangle\,.
\] 
By the definition of $\tau_M$ and \eqref{eq: KKK } we get  $f(i,j,\nu)=f_M(i,j, \nu)$ whenever $i,j \in B_M(0)$. 
 \end{proof}

We remark that, in general, it appears to be difficult to approximate the functions $g_h$  defined after  \eqref{eq: bijoint} from below by conservative and bounded vector fields  on the \emph{entire} $\R^d$. We now show that certain symmetric biconvex functions are $BD$-elliptic.

\begin{proposition}[Biconvexity implies $BD$-ellipticity]\label{prop: biconvex}
 Let  $f:\R^d\times \R^d\times \R^d\rightarrow [0,+\infty)$  be a symmetric biconvex function such that the associated $\theta$ satisfies $\lbrace \theta =0 \rbrace = \lbrace 0 \rbrace$. Then, $f$   is   $BD$-elliptic. 
\end{proposition}


In the proof, we will use the following technical property,  \BBB which applies in general for $GSBD^p$ functions. 

 \begin{lemma}\label{lemma: technical}
 Let $B \subset \R^d$ be open and bounded.  Suppose $u\in GSBD^p(B)$ such that $\int_{J_u} |[u]| \,  {\rm d}\mathcal{H}^{d-1} <+ \infty$. Then $u \in SBD^p(B)$.    
 \end{lemma}

\begin{proof}
The result follows from \cite[Theorem~2.9]{Crismale3} for $\mathbb{A}v={\rm E} v$ (see \cite[Remark~2.5]{Crismale3}). 
\end{proof}

\EEE

We proceed with the proof of Proposition \ref{prop: biconvex}.

\begin{proof}[Proof of Proposition \ref{prop: biconvex}]
Let us fix $(i,j,\nu)\in \R^d\times \R^d \times \mathbb{S}^{d-1}$  with $i\neq j$.  Consider  $v\in PR(Q^\nu_1)$ with $\{v\neq u_{i,j,\nu}\}\subset\!\subset Q^\nu_1$, where $u_{i,j,\nu}$ is the function  defined in (\ref{u_0}).  Without restriction, we may assume that $\int_{J_v} \theta\big([v]\odot \nu_v\big) \, {\rm d}\mathcal{H}^{d-1} <\infty$. Since $\theta$  is    positively $1$-homogeneous with  $\lbrace \theta =0 \rbrace = \lbrace 0 \rbrace$, we get $\theta(F) \ge c|F|$ for all $F \in \R^{d \times d}_{\rm sym}$ for some $c>0$. This implies 
$$\int_{J_v} |[v]| \, {\rm d}\mathcal{H}^{d-1} \le \sqrt{2} \int_{J_v} \big|[v]\odot \nu_v\big| \, {\rm d}\mathcal{H}^{d-1} \le \sqrt{2}/c \int_{J_v} \theta\big([v]\odot \nu_v\big) \, {\rm d}\mathcal{H}^{d-1} < + \infty. $$
Therefore, by Lemma \ref{lemma: technical} we get $v \in SBD(Q^\nu_1)$. In particular, the  symmetric distributional derivative $Ev$ of $v$ is a finite Radon measure,    and is  given by
\begin{align*}
Ev(B)=\int_{J_v\cap B} \big([v]\odot \nu_v\big) \, {\rm d}\mathcal{H}^{d-1}
\end{align*}
for all Borel sets $B \subset Q^\nu_1$. Since  $\{v\neq u_{i,j,\nu}\}\subset\!\subset Q^\nu_1$, we have   $Ev(Q^\nu_1)=Eu_{i,j,\nu}(Q^\nu_1)$. This along with Jensen's inequality and the fact that $\theta$ is positively $1$-homogeneous  and convex  yields
\begin{align*}
\int_{J_v} f(v^+,v^-,\nu_v)\, {\rm d}\mathcal{H}^{d-1}& = \int_{J_v}\theta\big([v]\odot\nu_v\big)\, {\rm d}\mathcal{H}^{d-1}\geq \theta\left( \int_{J_v} \big([v]\odot\nu_v\big) \, {\rm d}\mathcal{H}^{d-1} \right)\\&= \theta\left( \int_{J_{u_{i,j,\nu}}} \big( [u_{i,j,\nu}]\odot\nu \big)\, {\rm d}\mathcal{H}^{d-1} \right)= \mathcal{H}^{d-1}\big(J_{u_{i,j,\nu}} \cap Q^\nu_1\big)\, \theta\big((i-j)\odot\nu\big) \\&= f (i,j,\nu).
\end{align*}
This  shows that $f$ is  $BD$-elliptic and concludes the proof.
\end{proof}

\subsection{Independence of the traces at the jump}

 In this subsection, we consider functions which are independent of the traces at the jump set and only dependent on the normal, i.e.,
\begin{align}\label{eq: same}
 f(i,j,\nu) = \psi(\nu) \quad \quad \text{for all } (i,j,\nu) \in \R^d\times \R^d \times \BBB \R^d \quad  \text{with } i \neq j, \EEE
 \end{align} 
 for $\psi\colon\R^d \to [0,+\infty)$  even,  positively  $1$-homogeneous,  and convex. In this setting, it turns out that the notions of $BV$-ellipticity and $BD$-ellipticity  coincide. Recall that convexity of $\psi$ is a necessary condition for $BV$-ellipticity, see  \cite[Theorem 5.11, Theorem 5.14]{Ambrosio-Fusco-Pallara:2000}, and thus also necessary for $BD$-ellipticity. 

\begin{proposition}\label{prop: only nu} 
A function $f\colon \R^d\times \R^d \times \mathbb{S}^{d-1} \to [0,+\infty)$ of the form \eqref{eq: same} is $BD$-elliptic  if $\psi\colon\R^d \to [0,+\infty)$ is even,   positively  $1$-homogeneous,  and convex.
\end{proposition}

\begin{proof}
Fix $(i,j,\nu)\in \R^d\times \R^d \times \mathbb{S}^{d-1}$, $i \neq j$, and consider $v\in PR(Q^\nu_1)$ with $\{v\neq u_{i,j,\nu}\}\subset\!\subset Q^\nu_1$, where $u_{i,j,\nu}$ is defined in (\ref{u_0}). In view of \eqref{eq: PR-def}, it is elementary to see that there exists $u \in PC(Q^\nu_1)$ with $\{u\neq u_{i,j,\nu}\}\subset\!\subset Q^\nu_1$ such that $\mathcal{H}^{d-1} (J_u \triangle J_v) = 0$. This along with the fact that $f$ is $BV$-elliptic (see \cite[Example 5.23]{Ambrosio-Fusco-Pallara:2000}) yields
\begin{align*}
\int_{J_v} f(v^+,v^-,\nu_v)\, {\rm d}\mathcal{H}^{d-1}& = \int_{J_u} f(u^+,u^-,\nu_u)\, {\rm d}\mathcal{H}^{d-1}  \ge f (i,j,\nu).
\end{align*}
This concludes the proof.
  \end{proof}

  By choosing $\psi \equiv 1$  on $\mathbb{S}^{d-1}$,  we re-derive the well-known fact that the Hausdorff-measure $\mathcal{H}^{d-1}$ is lower semicontinuous on $PR(\Omega)$. Moreover,  we briefly remark that lower semicontinuity of functionals with integrands of the form \eqref{eq: same} has already been addressed in  \cite[Corollary 5.5]{Crismale-Friedrich}   in the setting of $GSBD^p$ functions.  As an alternative proof of $BD$-ellipticity, we can check that functions of the above kind are symmetric jointly convex. \EEE

\begin{proposition}\label{prop: only nu2} 
A function $f\colon \R^d\times \R^d \times \BBB \R^d \EEE \to [c,+\infty)$, $c>0$, of the form \eqref{eq: same} is symmetric jointly convex  if $\psi\colon\R^d \to [0,+\infty)$ is even,   positively  $1$-homogeneous,  and convex.
\end{proposition}

\begin{proof}
As $\psi$ is  even,  positively  $1$-homogeneous, convex, and bounded away from zero on  $\mathbb{S}^{d-1}$,  we find a bounded, open,  convex set  $K$  being symmetric with respect to the origin (i.e., $q\in K$ if and only if $-q\in K$) such that 
\begin{align}\label{eq: Krepr}
\psi(x)= \sup_{q\in K} \langle x,  q \rangle  = \sup_{q\in K} |\langle x,  q \rangle| \quad \quad \text{for all $x \in \R^d$}.
\end{align} 
Let us define the functions $g_{p,q,h}\colon\R^d \rightarrow \R^d$ by
\begin{align*}
g_{p,q,h}(i):= \theta_h\big( \langle i-p, q\rangle \big) \, q,
\end{align*}
where  $p\in \mathbb{Q}^d$, $q\in \mathbb{Q}^d \cap K$,  and for each  $h \in \N$ the function $\theta_h:\R \rightarrow [0,1]$ is given by 
\begin{align}\label{eq: tttheta}
\theta_h(y):=   \min\lbrace h|y|, 1 \rbrace.
\end{align}
Clearly, each $g_{p,q, h}$ is uniformly continuous, bounded, and conservative. Our goal is to prove that   for all $(i,j,\nu) \in \R^d\times \R^d\times \R^d$ with $i \neq j$ we have   
\begin{align}\label{eq: new-equ}
f(i,j,\nu) = \psi(\nu) = \sup_{p,q,h}\left\langle g_{p,q,h}(i)- g_{p,q,h}(j),\nu  \right\rangle.
\end{align}

\noindent \emph{Step 1: Proof of ``$\ge$''.}
For each   $(i,j,\nu) \in \R^d\times \R^d\times \R^d$ with $i \neq j$, we get immediately from \eqref{eq: Krepr} and \eqref{eq: tttheta} that 
\begin{align*}
\left\langle g_{p,q,h}(i)- g_{p,q,h}(j),\nu  \right\rangle \le |\langle q, \nu\rangle| \le  \sup_{q\in K} |\langle   q,\nu \rangle| = \psi(\nu) = f(i,j,\nu).
\end{align*}
By passing to the sup on the left hand side, we obtain  ``$\ge$'' in \eqref{eq: new-equ}.

\noindent \emph{Step 2: Proof of ``$\le$''.} Fix  $(i,j,\nu) \in \R^d\times \R^d\times \R^d$ with $i \neq j$ and let $\eps >0$. First, by \eqref{eq: Krepr} and the fact that $K$ is open we choose $q \in K \cap \mathbb{Q}^d$ such that $\langle q, \nu \rangle > \psi(\nu) -\eps$ and $\langle q, i-j\rangle \neq 0$. Then, for $h \in \N$ sufficiently large we have $|\langle q, i-j\rangle| \ge 1/h$. Therefore, for $p=j$ we get  by \eqref{eq: tttheta}
$$\left\langle g_{p,q,h}(i)- g_{p,q,h}(j),\nu  \right\rangle =   \big\langle\big(\theta_h\big( \langle i-j, q\rangle \big) - \theta_h(0) \big) \, q, \nu \big\rangle = \langle q, \nu \rangle > \psi(\nu) -\eps. $$    
By a density argument and the continuity of $\theta_h$ this still holds if $p$ is chosen in $\mathbb{Q}^d$. The arbitrariness of $\eps$ yields ``$\le$'' in \eqref{eq: new-equ}. This concludes the proof. 
\end{proof}

\EEE

\subsection{Functions with mild dependence on the traces}\label{sec:O}

In this subsection, we consider another class of $BD$-elliptic functions 
$$
f(i,j,\nu)= g(i-j)\quad\quad \text{for all }  (i,j,\nu)\in \R^d\times\R^d\times\mathbb{S}^{d-1},
$$
where $g\colon\R^d\rightarrow [0,+\infty)$ is a bounded, even function with 
\begin{align}\label{eq: alphabeta}
  \sup g \le 2\inf g.
\end{align}
Due to \eqref{eq: alphabeta}, we say that $f$ has only a \emph{mild dependence} on the traces at the jump.

\begin{proposition}
Under \eqref{eq: alphabeta}, the function $f$ is $BD$-elliptic.
\end{proposition}

\begin{proof}
 Fix  $(i,j,\nu)\in \R^d\times\R^d\times\mathbb{S}^{d-1}$ with $i \neq j$. 
 Consider $u \in PR(Q^\nu_1)$ such that $\{u\neq u_{i,j,\nu}\}\subset\!\subset Q^\nu_1$, where $u_{i,j,\nu}$ is the function  defined in (\ref{u_0}). In view of \eqref{eq: PR-def}, we can write $u  =  \sum\nolimits_{k\in \N} a_{Q_k,b_k} \chi_{P_k}$, where $P_1 = \lbrace u=i\rbrace$ and $P_2 = \lbrace u =j\rbrace$, and $\mathcal{H}^{d-1}(J_u \triangle (\bigcup_{k \ge 1} \partial^* P_k \setminus \partial Q^\nu_1   ) )=0 $. We define 
 $$\Gamma_1 = \partial^* P_1 \cap \partial^* P_2, \quad \quad \quad \Gamma_2 = \bigcup_{k \ge 3}\partial^* P_k \setminus \Gamma_1.$$
  The local structure of Caccioppoli partitions (see \cite[Theorem 4.17]{Ambrosio-Fusco-Pallara:2000}) and the fact that $\partial^* P_k \cap \partial Q^\nu_1  = \emptyset$ for $k \ge 3$ imply that $J_u = \Gamma_1 \cup \Gamma_2$ up to a set of $\mathcal{H}^{d-1}$-negligible measure.   We  now  introduce some more notation. We define $\Pi_\nu = \lbrace x\in \R^d \colon \langle x,\nu \rangle = 0 \rbrace$ and, for $B \subset \R^d$, we let $B^\nu_y = \lbrace t\in\R\colon y + t\nu \in B \rbrace$ for each $y \in \Pi_\nu$.   We decompose the set $\Pi_\nu \cap Q^\nu_1$ into the sets 
$$T_1 =  \big\{ y \in \Pi_\nu \cap Q^\nu_1\colon (\Gamma_1)^\nu_y \neq \emptyset\big\}, \quad \quad T_2 = (\Pi_\nu \cap Q_1^\nu) \setminus T_1.$$
 As $\{u\neq u_{i,j,\nu}\}\subset\!\subset Q^\nu_1$, we find that each line $y + \R \nu$, $y \in \Pi_\nu \cap Q^\nu_1$, intersects $P_1$ and $P_2$ on a set of positive Lebesgue measure. Thus,  for $\mathcal{H}^{d-1}$-a.e. $y \in \Pi_\nu \cap Q^\nu_1$,  by  slicing properties \cite[Theorem~3.108]{Ambrosio-Fusco-Pallara:2000} for the $BV$ functions $\chi_{P_1}$ and $\chi_{P_2}$, we get  $\mathcal{H}^0\big((y + \R \nu) \cap \partial^* P_k\big) \ge 1$ for $k=1,2$. Then, the local structure of Caccioppoli partitions (see \cite[Theorem 4.17]{Ambrosio-Fusco-Pallara:2000}) implies that  for $\mathcal{H}^{d-1}$-a.e. $y \in T_2$ there exist other components $P_k, P_l$ for some $k,l \ge 3$ (possibly $k=l$) such that the line  $y + \R \nu$  intersects $\partial^* P_1 \cap \partial^* P_k$ and $\partial^* P_2 \cap \partial^* P_l$. Thus, there holds  
\begin{align}\label{eq: later-bound}
\mathcal{H}^0\big( (\Gamma_2)^\nu_y \big) \ge 2 \quad \quad \text{for $\mathcal{H}^{d-1}$-a.e. $y \in T_2$}.
\end{align}
 As $u = i$ on $P_1$ and $u = j$ on $P_2$, we   obtain
\begin{align*}
\mathcal{F}(u)&= \int_{\Gamma_1} g([u])\,{\rm d}\mathcal{H}^{d-1} + \int_{\Gamma_2} g([u])\,{\rm d}\mathcal{H}^{d-1} \ge g(i-j) \mathcal{H}^{d-1}(\Gamma_1) + \inf g \, \mathcal{H}^{d-1}(\Gamma_2).  
\end{align*}
 By $\nu_{\Gamma_1}$ and $\nu_{\Gamma_2}$ we denote unit normals to the rectifiable sets $\Gamma_1$ and $\Gamma_2$, respectively.  By the area formula (cf.\ e.g.\ \cite[(12.4) in Section~12]{Sim84}) and by \eqref{eq: later-bound} there holds
\begin{align*}
\mathcal{F}(u)&\ge g(i-j) \int_{\Gamma_1} |\langle \nu,\nu_{\Gamma_1} \rangle|\, {\rm d} \mathcal{H}^{d-1} + \inf g \, \int_{\Gamma_2} |\langle \nu,\nu_{\Gamma_2} \rangle|\, {\rm d} \mathcal{H}^{d-1} \\
&  \ge  g(i-j) \int_{T_1}  \,  \mathcal{H}^0\big( (\Gamma_1)^\nu_y \big)   \,    {\rm d} \mathcal{H}^{d-1}(y)   + \inf g \, \int_{T_2}    \mathcal{H}^0\big( (\Gamma_2)^\nu_y \big)     \,  {\rm d} \mathcal{H}^{d-1}(y) \\
& \ge g(i-j)\, \mathcal{H}^{d-1}(T_1) + 2\inf g \, \mathcal{H}^{d-1}(T_2).
\end{align*}
By \eqref{eq: alphabeta} and the fact that $\mathcal{H}^{d-1}(T_1) + \mathcal{H}^{d-1}(T_2) = 1$, we conclude $\mathcal{F}(u) \ge g(i-j) = f(i,j,\nu)$. 
\end{proof}

\subsection{$BV$-elliptic, but not $BD$-elliptic functions}\label{sec: counterexamples}

In this subsection, we provide two examples of $BV$-elliptic functions which are not $BD$-elliptic.

\begin{example}[Anisotropy in jump normal]\label{counterexample}
Consider functions $f\colon\R^d\times\R^d\times  \mathbb{S}^{d-1}  \rightarrow [0,+\infty)$ of the form 
\begin{align*}
f(i,j,\nu):= |i-j|\,\psi(\nu),
\end{align*}
where $\psi:\R^d\rightarrow [0,+\infty)$ is convex, even, and   positively  $1$-homogeneous. Recall from \eqref{eq: BV class} that densities of this form are $BV$-elliptic.  We show that $f$ is in general not $BD$-elliptic if $\psi$ is anisotropic. To see this, we let  $d=2$ for simplicity and suppose that $\nu = e_2$, where $\lbrace e_1, e_2\rbrace$ denotes the standard orthonormal basis of $\R^2$. Assume that $\psi(e_1) = \eps$ and $\psi(e_2) = 1$ for some $\eps>0$ small to be specified later. \BBB For notational simplicity, we consider functions defined on $Q^{e_2}_6$. Let $i=(0,0)$, \EEE $j=(2\lambda,2\lambda)$ for $\lambda >0$, and let $u\in PR(Q_6^{e_2})$ be defined by  
\begin{align*}
u(x):=
\begin{cases}
 u_{i,j,e_2}(x)  \quad & \text{on } Q^{e_2}_6\setminus Q^{e_2}_{2},\\
  a_{Q,b}(x)  \quad & \text{on }Q^{ e_2}_{2},
\end{cases}
\end{align*}
where $ Q  =    \lambda (e_1 \otimes e_2  - e_2 \otimes e_1)  \in \R^{2\times 2}_{\rm skew}$ and $b = (\lambda,\lambda)\in\R^2$. The affine function is chosen in such a way that the set of discontinuities of the scalar functions $u_1 = \langle u, e_1 \rangle $ and $u_2=\langle u, e_2\rangle $ is the one represented in Figure \ref{fig:counterexample1}.
\begin{figure}[!htb]
\centering
\def\svgwidth{13cm}
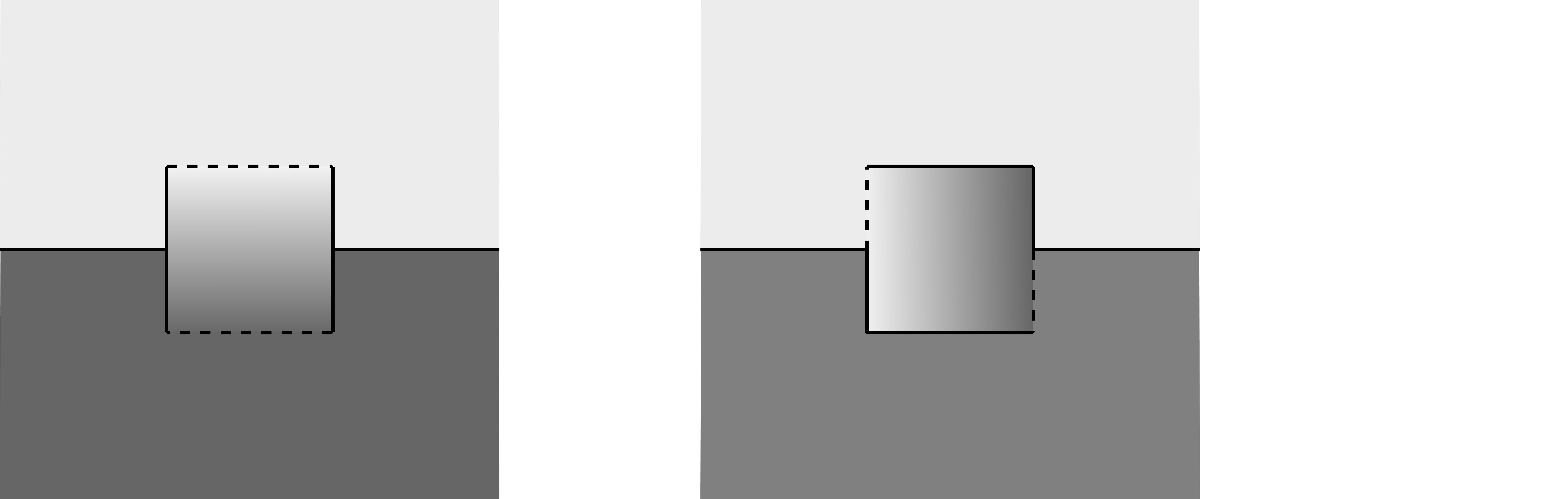
\caption{The lines (not dashed) in the above pictures are a pictorial representation of the set of discontinuities of the function $u_1$ (on the left) and of the function $u_2$ (on the right). The dashed bold lines correspond to points where the function is not discontinuous.}
\label{fig:counterexample1}
\end{figure}
We define 
\begin{align}\label{eq: ortho definition} 
J_u^\|:= \{x\in J_u\colon\,|\langle \nu_u(x),  e_2  \rangle |= 1  \}, \ \ \ \ \ \ \ J_u^\perp:= \{ x\in J_u\colon\, \langle \nu_u(x),  e_2\rangle=0  \}.
\end{align}
Up to a set of negligible $\mathcal{H}^1$-measure, we can write $J_u^{\|}$ as the union of four pairwise disjoint sets $\Gamma_k$  with $\mathcal{H}^1(\Gamma_k)=2$ for $k=1,\dots,4$, \BBB namely  
\begin{align*} 
&\Gamma_1= [-3,-1] \times \lbrace 0 \rbrace, \ \ \ \ \ \ \ \Gamma_2=  \lbrace -1 \rbrace \times [-1,1],  \ \ \ \ \ \ \ \Gamma_3=  \lbrace 1 \rbrace \times [-1,1], \ \ \ \ \ \ \ \Gamma_4= [1,3] \times \lbrace 0 \rbrace, 
\end{align*}
see Figure \ref{fig:counterexample1}. \EEE Then, by $\psi(e_1) = \eps$ and $\psi(e_2) = 1$ we get 
\begin{align*}
\mathcal{F}(u,Q_6^{e_2})&=\int_{J_u}|[u]|\psi(\nu_u)\,{\rm d}\mathcal{H}^{1}=\int_{J_u^\perp}|[u]|\psi(e_1) \,{\rm d}\mathcal{H}^{1} + \int_{J_u^\|}|[u]| \psi(e_2)\,{\rm d}\mathcal{H}^{1} \le C\lambda \eps +  \int_{J_u^\|}|[u]| \,{\rm d}\mathcal{H}^{1}
\end{align*} 
for a universal $C>0$. Since $|[u]| = 2\sqrt{2}\lambda$ on $\Gamma_1 \cup \Gamma_4$, as well as  $a_{Q,b}(t,-1) = \lambda(1-t)e_2$ and $a_{Q,b}(t,1) = 2\lambda e_1 + \lambda(1-t)e_2$ for $t \in (-1,1)$, a direct computation shows 
\begin{align*}
 \int_{J_u^\|}|[u]| \,{\rm d}\mathcal{H}^{1} & =   \int_{\Gamma_1 \cup \Gamma_4}|[u]| \,{\rm d}\mathcal{H}^{1} +  \int_{\Gamma_2 \cup \Gamma_3}|[u]| \,{\rm d}\mathcal{H}^{1} = 4 \cdot 2\sqrt{2}\lambda +  \int_{\Gamma_2 \cup \Gamma_3}|[u]| \,{\rm d}\mathcal{H}^{1} \\
 & =  8\sqrt{2}\lambda +  \int_{-1}^{1}  \lambda(1-t)   \, {\rm d}t  +  \int_{-1}^{1}  \lambda(1+t)   \, {\rm d}t = (8\sqrt{2} + 4)\lambda. 
\end{align*}
Thus, $\mathcal{F}(u,Q_6^{e_2}) \le C\lambda\eps + (8\sqrt{2} + 4) \lambda< 12\sqrt{2}\lambda$ for $\eps$ small enough.  Observing  that  $\mathcal{F}(u_{i,j,e_2},Q_6^{e_2}) = 6 \cdot 2\sqrt{2}\lambda = 12\sqrt{2}\lambda$, we find $\mathcal{F}(u,Q_6^{e_2}) < \mathcal{F}(u_{i,j,e_2},Q_6^{e_2})$. This shows that $f$ is not $BD$-elliptic and concludes the example.
\end{example}

\begin{example}[Anisotropy in jump direction]\label{counterexample2}
Consider functions $f:\R^d\times\R^d\times \mathbb{S}^{d-1} \rightarrow [0,+\infty)$ of the form 
\begin{align*}
f(i,j,\nu):= \psi(i-j),
\end{align*}
where $\psi:\R^d\rightarrow [0,+\infty)$ is a norm on $\R^d$.  Recall from \eqref{eq: BV class} that densities of this form a $BV$-elliptic.  We show that $f$ is in general not $BD$-elliptic if $\psi$ is anisotropic. We again consider $d=2$,  $\nu = e_2$, and define  $\psi(x_1,x_2)=\sqrt{x_1^2 + \eps x_2^2}$ for $(x_1,x_2) \in \R^2$ for some $\eps>0$ small to be specified later.  Let $i=(0,0)$ and $j=(2\lambda,2\lambda)$ for $\lambda >0$. For  $\delta = \eps^{1/4}  >0$,  we let $u\in PR(Q_6^{e_2})$ be given by  
\begin{align*}
u(x):=
\begin{cases}
 u_{i,j,e_2}(x)  \quad & \text{on } Q^{e_2}_6\setminus R_\delta,\\
 a_{Q,b}(x)  \quad & \text{on } R_\delta,
\end{cases}
\end{align*}
where $R_\delta := (-1,1) \times (-\delta,\delta)$,   $Q =    \lambda/\delta (e_1 \otimes e_2  - e_2 \otimes e_1)  \in \R^{2\times 2}_{\rm skew}$ and  $b = (\lambda,\lambda/\delta)\in\R^2$.  We define $J_u^\|$ and $J_u^\perp$ as in \eqref{eq: ortho definition} and again note that, up to a set of negligible $\mathcal{H}^1$-measure,  $J_u^{\|}$ consists of four pairwise disjoint sets $\Gamma_k$  with $\mathcal{H}^1(\Gamma_k)=2$ for $k=1,\dots,4$. Since $\Vert \langle a_{Q,b}, e_1 \rangle \Vert_{L^\infty(R_\delta)} \le C\lambda$ and $\Vert  \langle a_{Q,b},  e_2 \rangle \Vert_{L^\infty(R_\delta)} \le C\lambda/\delta$   for a universal $C>0$, we get   $|[u_1]|\le C\lambda$ and  $|[u_2]|\le C\lambda/\delta$  on $J_u^\perp$,  where $u_k := \langle u, e_k \rangle$, $k=1,2$.   Therefore,   we obtain
\begin{align}\label{eq: int step}
\mathcal{F}(u,Q_6^{e_2})&=\int_{J_u} \psi([u])  \,{\rm d}\mathcal{H}^{1}=\int_{J_u^\perp}  \psi([u])  \,{\rm d}\mathcal{H}^{1} + \int_{J_u^\|}  \psi([u])  \,{\rm d}\mathcal{H}^{1} \\& \le C\mathcal{H}^1(J_u^\bot)   \sqrt{\lambda^2 + \eps (\lambda/\delta)^2}  +  \int_{J_u^\|}  \psi([u])  \,{\rm d}\mathcal{H}^{1} \le C\lambda\delta\sqrt{1+\eps/\delta^2} +  \int_{J_u^\|}  \psi([u])  \,{\rm d}\mathcal{H}^{1}, \notag
\end{align} 
 where we have also used that $\mathcal{H}^1(J_u^\bot)\le 4  \delta $.  Since $ \psi([u])  = 2\sqrt{1+\eps}\lambda$ on $\Gamma_1 \cup \Gamma_4$, as well as  $a_{Q,b}(t, -\delta)  = \frac{\lambda}{\delta}(1-t)e_2$ and $a_{Q,b}(t, \delta)  = 2\lambda e_1 + \frac{\lambda}{\delta}(1-t)e_2$ for $t \in (-1,1)$, we compute 
\begin{align*}
 \int_{J_u^\|}  \psi([u])  \,{\rm d}\mathcal{H}^{1} & =   \int_{\Gamma_1 \cup \Gamma_4}  \psi([u]) \,{\rm d}\mathcal{H}^{1} +  \int_{\Gamma_2 \cup \Gamma_3}  \psi([u]) \,{\rm d}\mathcal{H}^{1} = 4 \cdot 2\sqrt{1+\eps}\lambda +  \int_{\Gamma_2 \cup \Gamma_3}  \psi([u]) \,{\rm d}\mathcal{H}^{1} \\
 & =  8\sqrt{1+\eps}\lambda +  \int_{-1}^{1}   \sqrt{\eps}  \frac{\lambda}{\delta}(1-t)   \, {\rm d}t  +  \int_{-1}^{1}  \sqrt{\eps}   \Big|2\lambda -   \frac{\lambda}{\delta}(1-t) \Big|    \, {\rm d}t \le 8\lambda + C\lambda \sqrt{\eps}  /\delta. 
\end{align*}
In view of \eqref{eq: int step} and   $\delta = \eps^{1/4}$,  we thus get $\mathcal{F}(u,Q_6^{e_2}) \le 8\lambda + C\lambda   \eps^{1/4} $. By choosing $\eps$ sufficiently small, we find $\mathcal{F}(u,Q_6^{e_2}) < 6\cdot 2\lambda \sqrt{1+\eps} =  \mathcal{F}(u_{i,j,e_2},Q_6^{e_2})$. This shows that $f$ is not $BD$-elliptic and concludes the example.
\end{example}

\BBB We close this subsection by noting that $BV$-elliptic functions in one dimension induce $BD$-elliptic functions in the following sense: given a $BV$-elliptic function $f:\R\times\R\times \R\mapsto [0,+\infty)$ which is positively $1$-homogeneous in the third variable, we observe that the function
$$\bar{f}(i,j,\nu) := \sup_{\xi \in \mathbb{S}^{d-1}} f\big( \langle i, \xi \rangle,  \langle j, \xi \rangle, \langle \nu, \xi \rangle  \big)$$
is $BD$-elliptic. This follows directly from the definition of $BV$- and $BD$-ellipticity along with a slicing argument and the fact that, given $u \in PR(Q^\nu_1)$, the function $t \mapsto \langle  u(t \xi), \xi \rangle$ is piecewise constant. Lower semicontinuity of such densities has already been addressed in \cite{GZ3}. \EEE

\section{Lower semicontinuity in $GSBD^p$ for symmetric jointly convex functions}\label{sec: GSBD}

This section is devoted to a lower semicontinuity result for surface integrals in $GSBD^p(\Omega)$, $p>1$, where the integrands are symmetric jointly convex functions, see Definition \ref{def: symm-conv}. We also discuss well-posedness of certain minimization problems.  We refer to \cite{DM} for the definition and the properties of this function space.

\begin{theorem}[Lower semicontinuity of surface integrals in $GSBD^p$]\label{thm: GSBD LSC}
Let $f:\R^d\times\R^d\times \R^d \rightarrow [0,+\infty)$  be a symmetric jointly convex function. Then, for every sequence $(u_k)_k\subset GSBD^p(\Omega)$, $p>1$, converging in measure to $u\in GSBD^p(\Omega)$, and satisfying the condition
\begin{align}\label{eq: e(u_k) unif bdd}
\sup_{k\in\mathbb{N}} \big( \| e(u_k) \|_{L^p(\Omega)}  + \mathcal{H}^{d-1}(J_{u_k}) \big) <+\infty,
\end{align}
we have that 
\begin{align}\label{eq: GSBD LSC}
\int_{J_u} f(u^+,u^-,\nu_u)\, {\rm d}\mathcal{H}^{d-1} \leq \liminf_{k\to \infty} \int_{J_{u_k}} f(u_k^+,u_k^-, \nu_{u_k})\, {\rm d}\mathcal{H}^{d-1}.
\end{align}
\end{theorem}

For various examples of symmetric jointly convex integrands we refer the reader to Section \ref{sec: examples}.  Restricting to the space of $SBD$-functions, the above  result  also holds for symmetric biconvex functions introduced in Subsection \ref{sec: biconvex}, see Theorem \ref{thm: biconv} below for  details.  As a consequence of the above result, we get that the following minimization problems are well-posed.

\begin{theorem}[Existence of minimizers]\label{thm: min}
Let $\Omega \subset \R^d$ be open and bounded and let $c>0$. Let $W\colon \R^{d\times d}_{\rm sym} \to [0,+\infty)$ be convex with $W(F) \ge c|F|^p$ for all $ F \in  \R^{d \times d}_{\rm sym}$ for some $p>1$. Let $f\colon \R^d \times \R^d \times \mathbb{S}^{d-1} \to [c,+\infty)$ be symmetric jointly convex and let $\Psi: [0,+\infty) \to [0,+\infty)$ be continuous such that $\lim_{t \to \infty} \Psi(t) = + \infty$. Then, the functional
$$\mathcal{E}(u) :=  \int_\Omega W\big(e (u) \big) \, {\rm d}x + \int_{J_u} f(u^+,u^-,\nu) \, {\rm d}\mathcal{H}^{d-1} + \int_\Omega \Psi(|u|) \, {\rm d}x \quad \quad \text{for all } u \in GSBD^p(\Omega)$$    
has a minimizer in $GSBD^p(\Omega)$.
\end{theorem}

\begin{proof}[Proof of Theorem \ref{thm: min}]
Let $(u_k)_k \subset GSBD^p(\Omega)$ be a minimizing sequence. Then, by the growth condition of $W$ and the fact that $f \ge c >0$, we obtain 
\begin{align}\label{eq: energy bound}
\sup_{k \in \N} \big( \Vert e(u_k) \Vert_{L^p(\Omega)} + \mathcal{H}^{d-1}(J_{u_k}) \big) < + \infty.
\end{align}
This along with the fact that $\sup_{k \in \N} \int_\Omega \Psi(|u_k|) {\rm d}x < + \infty$ allows us to  apply \cite[Theorem 11.3]{DM}:  we find $u\in GSBD^p(\Omega)$ such that $u_k \to u$ a.e.\ on $\Omega$ and $e(u_k) \rightharpoonup e(u)$ weakly in $L^p(\Omega;\R^{d\times d}_{\rm sym})$. By the convexity of $W$ and Fatou's lemma we obtain
$$ \int_\Omega W\big(e (u) \big) \, {\rm d}x \le \liminf_{k\to \infty}  \int_\Omega W\big(e (u_k) \big) \, {\rm d}x, \quad \quad \quad \int_\Omega \Psi(|u|) \, {\rm d}x \le \liminf_{k \to \infty} \int_\Omega \Psi(|u_k|) \, {\rm d}x.$$
By \eqref{eq: energy bound} and Theorem \ref{thm: GSBD LSC} we also get that the surface term is lower semicontinuous. We thus conclude that $u$ is a minimizer.  
\end{proof}
 
 The remainder of the section is devoted to the proof of Theorem \ref{thm: GSBD LSC}.    The proof   will rely on the following integration by parts formula.

\begin{lemma}[Integration by parts in $GSBD$]\label{lemma: BD-int-by-parts}
Let  $G \in C^1(\R^d;\R^d) \cap W^{1,\infty}(\R^d;\R^d)$  be conservative and let $u \in GSBD^p(\Omega)$, $1 < p < \infty$. Then, for all $A \in \mathcal{A}(\Omega)$ and $\varphi \in C^1_c(A)$  there holds
\begin{align}\label{eq: chaini}
\int_{A\cap J_u} \left\langle G(u^+)-G(u^-),\nu_u  \right\rangle \varphi\,{\rm d}\mathcal{H}^{d-1} + \int_A \big((\nabla G(u): e(u)  \big)\, \varphi\, {\rm d}x = - \int_A \left\langle G(u),\nabla \varphi  \right\rangle\,{\rm d}x.
\end{align}
\end{lemma}

In order to prove this formula,  we will combine the corresponding formula in $SBV$ (see \cite[Lemma~3.5]{Amb}) with  an approximation result for $GSBD^p$ functions  stated in  \cite[Theorem 1.1]{Crismale2}. A slightly simplified statement of  the latter  result is the following.

\begin{theorem}[Density in $GSBD^p$]\label{thm: CCapprox}
Let $u\in GSBD^p(\Omega)$, $p>1$. Then, there exists a sequence of functions $(u_k)_k\subset SBV^p(\Omega;\R^d)\cap L^\infty(\Omega;\R^d)$  such that each $J_{u_k}$ is closed in $\Omega$ and included in a finite union of closed connected pieces of $C^1$ hypersufaces, $u_k\in W^{1,\infty}(\Omega\setminus J_{u_k};\R^d)$, and 
\begin{align}\label{eq: main properties}
{\rm (i)} & \ \ u_k \to u \text{ a.e.\ on $\Omega$},\notag\\
{\rm (ii)} & \ \  \|e(u_k) - e(u) \|_{L^p(\Omega) }  \to 0,\notag\\
{\rm (iii)} & \ \ \mathcal{H}^{d-1}(J_{u_k}\triangle J_u)\to 0,\notag \\ 
{\rm (iv)} & \ \  \int_{J_{u_k} \cup J_u} \tau (|u_k^\pm- u^\pm|) \,{\rm d}\mathcal{H}^{d-1} \to 0, 
\end{align}
for some $\tau \in C^1(\R )$ with $-\frac{1}{2} \le \tau \le \frac{1}{2}$, $0 \le \tau' \le 1$,  and $\lbrace \tau=0 \rbrace = \lbrace 0\rbrace$.  
\end{theorem}

\begin{proof}[Proof of Lemma \ref{lemma: BD-int-by-parts}]
Given $u\in GSBD^p(\Omega)$ with $p>1$, we let $(u_k)_k\subset SBV^p(\Omega;\R^d)\cap L^\infty(\Omega;\R^d)$ be the approximation sequence provided by Theorem \ref{thm: CCapprox}. Then, by \cite[Lemma~3.5]{Amb}, for all $A \in \mathcal{A}(\Omega)$, for every $\varphi \in C^1_c(A)$, and for all $k\in\mathbb{N}$ we have that
\begin{align}\label{eq: chaini_h}
\int_{A\cap J_{u_k}} \left\langle G(u_k^+)-G(u_k^-),\nu_{u_k}  \right\rangle \varphi\,{\rm d} \mathcal{H}^{d-1} + \int_A \big( (\nabla G(u_k))^T  : \nabla u_k \big) \, \varphi\, {\rm d}x = - \int_A \left\langle G(u_k),\nabla \varphi  \right\rangle\,{\rm d}x.
\end{align}
Our goal is to pass to the limit $k\to \infty$ in each of the three terms separately.

\noindent \emph{Step 1.}  As $G \in C(\R^d;\R^d) \cap L^\infty(\R^d;\R^d)$ and $\nabla \varphi \in L^\infty(A)$, we get by   \eqref{eq: main properties}(i) and dominated convergence that 
\begin{align}\label{eq: integration by parts step1}
\lim_{k\to\infty}\int_{A}\left\langle G(u_k),\nabla \varphi  \right\rangle \,{\rm d}x = \int_{A}\left\langle G(u),\nabla \varphi  \right\rangle \,{\rm d}x. 
\end{align}  
\noindent \emph{Step 2.} We now show that 
\begin{align}\label{eq: integration by parts step2}
\lim_{k\to\infty}\int_A \big((\nabla G(u_k))^T  : \nabla u_k  \big) \, \varphi\, {\rm d}x = \int_A \big(\nabla G(u): e(u)  \big)\, \varphi\, {\rm d}x. 
\end{align}
In fact, we first note that $(\nabla G(u_k))^T  : \nabla u_k = \nabla G(u_k): e(u_k)$ a.e.\ due to the fact that $G$ is a conservative vector field and thus $\nabla G\colon \R^d \to \R^{d\times d}_{\rm sym}$. As $\nabla G \in C(\R^d;\R^{d\times d}) \cap L^\infty(\R^d;\R^{d\times d})$, we get $\nabla G(u_k) \to \nabla G(u)$ in $L^q(\Omega;\R^{d\times d})$ for any $q\in [1,\infty)$ by \eqref{eq: main properties}(i) and dominated convergence. Thus, by \eqref{eq: main properties}(ii), $\varphi \in C^1_c(A)$, and H\"older's inequality we obtain \eqref{eq: integration by parts step2}.

\noindent \emph{Step 3.} We  finally prove that, up to a subsequence, there holds
\begin{align}\label{eq: integration by parts step3}
\lim_{k\to\infty} \int_{A\cap J_{u_k}} \left\langle G(u_k^+)-G(u_k^-),\nu_{u_k}  \right\rangle \varphi\, {\rm d}\mathcal{H}^{d-1}= \int_{A\cap J_{u}} \left\langle G(u^+)-G(u^-),\nu_{u}  \right\rangle \varphi\, {\rm d}\mathcal{H}^{d-1}.
\end{align} 
As a preliminary step, we observe that, up to a subsequence, 
\begin{align}\label{eq: two prop}
&\lim_{k\to \infty}u^{\pm}_k  (x)  \rightarrow u^{\pm}  (x)  \quad \text{for $\mathcal{H}^{d-1}$-a.e. }x\in A. \notag\\
&\lim_{k\to \infty} \nu_{u_k}(x)=\nu_u(x) \quad \text{for $\mathcal{H}^{d-1}$-a.e. }x\in A.
\end{align}
In fact, by \eqref{eq: main properties}(iv) we get that  $u_k^{\pm}$ converges to $u^{\pm}$ in measure with respect to the measure $\mathcal{H}^{d-1}$, i.e., 
$$\lim_{k\to \infty} \mathcal{H}^{d-1} \left(\left\{x\in (J_{u_k}\cup J_u)\cap A\colon\, |u_k^{\pm}  (x)  -u^{\pm}  (x)  |> \eps  \right\}  \right)= 0$$
for all $\eps >0$. Then, up to passing to a subsequence,  we get that \eqref{eq: two prop}(i) holds true. We  further  observe that  \eqref{eq: two prop}(ii) follows directly from   \eqref{eq: main properties}(iii). Now, by \eqref{eq: two prop}, dominated convergence, and the fact that $G \in C(\R^d;\R^d) \cap L^\infty(\R^d;\R^d)$ as well as $\varphi \in C(A)\cap L^\infty(A)$ we obtain \eqref{eq: integration by parts step3}.


Finally, taking into account \eqref{eq: integration by parts step1}, \eqref{eq: integration by parts step2}, and \eqref{eq: integration by parts step3}, we can pass to the limit in \eqref{eq: chaini_h}. This concludes the proof.
\end{proof}

We are now in a position to prove Theorem \ref{thm: GSBD LSC}.

\begin{proof}[Proof of Theorem \ref{thm: GSBD LSC}]
The proof is in the spirit of that of  \cite[Theorem 3.6]{Amb}.  The essential step is to show that   for every every $A\in \mathcal{A}(\Omega)$ and for every conservative vector field $G \in C^1(\R^d;\R^d) \cap W^{1,\infty}(\R^d;\R^d)$ there holds 
\begin{align}\label{eq: GSBD LSC intermediate}
\int_{J_u\cap A} \left\langle G(u^+)-G(u^-), \nu_{u}  \right\rangle^+\, {\rm d}\mathcal{H}^{d-1} \leq \liminf_{k\to \infty} \int_{J_{u_k}\cap A} \left\langle G(u_k^+)-G(u_k^-), \nu_{u_k}  \right\rangle^+ \, {\rm d}\mathcal{H}^{d-1},
\end{align}
 where $t^+ := \max\lbrace t,0 \rbrace$ for $t \in \R$.  
For the moment, we assume that  \eqref{eq: GSBD LSC intermediate} holds and show the statement (Step 1). Afterwards, we will prove  \eqref{eq: GSBD LSC intermediate} (Step 2).

\noindent \emph{Step 1: Proof of the statement.} Fix $\eps >0$. By Remark \ref{rem: symm-conv} and the fact that $f$ is nonnegative  there exist conservative vector fields  $(g_h^\eps)_h \subset  C^1(\R^d;\R^d)\cap W^{1,\infty}(\R^d;\R^d)$  such that
\begin{align}\label{eq: eps approx}
f(i,j,\nu)-\varepsilon \le \sup_{h \in \N} \left\langle g^{\varepsilon}_h(i)-g^{\varepsilon}_h(j),\, \nu\right\rangle^+ \le f(i,j,\nu)+\varepsilon  \quad  \text{for all } (i,j,\nu) \in \R^d\times \R^d\times \mathbb{S}^{d-1}, i\neq j.
\end{align}
For each $h \in \N$ and all $A\in \mathcal{A}(\Omega)$, we get by \eqref{eq: GSBD LSC intermediate} and \eqref{eq: eps approx}
 \begin{align*}
\int\limits_{J_u\cap A} \left\langle g_h^\eps(u^+)-g_h^\eps(u^-), \nu_{u}  \right\rangle^+\, {\rm d}\mathcal{H}^{d-1} \le  \liminf_{k\to \infty} \int\limits_{J_{u_k}\cap A} f(u_k^+,u_k^-, \nu_{u_k})\, {\rm d}\mathcal{H}^{d-1} + \eps \sup\nolimits_k  \mathcal{H}^{d-1}(J_{u_k}\cap A).
\end{align*}
We set
\begin{align}\label{eq: Lambda-eps}
\Lambda_\eps(A) :=  \liminf_{k\to \infty} \int_{J_{u_k}\cap A} f(u_k^+,u_k^-, \nu_{u_k})\, {\rm d}\mathcal{H}^{d-1} + \eps \sup\nolimits_k  \mathcal{H}^{d-1}(J_{u_k}\cap A)
\end{align}
and apply Lemma \ref{lem: localization} to find 
 \begin{align*}
\int_{J_u\cap A} \sup_{h \in \N} \  \left\langle g_h^\eps(u^+)-g_h^\eps(u^-), \nu_{u}  \right\rangle^+\, {\rm d}\mathcal{H}^{d-1} \le \Lambda_\eps(A) 
\end{align*}
for all $A \in \mathcal{A}(\Omega)$. Then, by \eqref{eq: eps approx} we get 
 \begin{align*}
\int_{J_u}  f (u^+,u^-,\nu_{u})\,{\rm d}\mathcal{H}^{d-1}  \le \Lambda_\eps(\Omega)  + \eps \mathcal{H}^{d-1}(J_u). 
\end{align*}
We conclude the proof of \eqref{eq: GSBD LSC} by passing to $\eps \to 0$ and using \eqref{eq: e(u_k) unif bdd} as well as \eqref{eq: Lambda-eps}.

\noindent \emph{Step 2: Proof of \eqref{eq: GSBD LSC intermediate}.}  We now show \eqref{eq: GSBD LSC intermediate}. By condition (\ref{eq: e(u_k) unif bdd}) we get that the sequence $(|e(u_k)|)_k$ is equiintegrable and so, for every $\eps>0$, we can find an open set $B\subset \Omega$ such that
\begin{align}\label{daje}
\sup_{k \in \mathbb{N}}\int_{B}|e(u_k)|\,{\rm d}x  + \int_{B}|e(u)|\,{\rm d}x\le \eps, \quad \quad \text{and } \quad  \bigcup_{k\in\mathbb{N}}J_{u_k} \cup J_u \subset B.
\end{align}
Let  $\Phi:=\{\varphi\in C^1_c(B)\colon\, 0\leq \varphi\leq 1 \}$.  Then, by (\ref{daje}), by  dominated convergence,  and by applying Lemma~\ref{lemma: BD-int-by-parts} twice,  we get
\begin{align*}
\int_{J_u\cap A}\left\langle G(u^+)-G(u^-), \nu_u \right\rangle^+\,{\rm d}\mathcal{H}^{d-1}&= \sup_{\varphi\in \Phi}\left\{ \int_{J_u\cap A}\left\langle G(u^+)-G(u^-), \nu_{u}\right\rangle \, \varphi\,{\rm d}\mathcal{H}^{d-1}   \right\}
\\&\leq C\eps + \sup_{\varphi\in\Phi}\left\{ -  \int_{A}\left\langle G(u),\nabla \varphi  \right\rangle\, {\rm d}x   \right\}
\\&\leq C\eps + \liminf_{k\to \infty}\sup_{\varphi\in\Phi}\left\{ -  \int_{A}\left\langle G(u_k),\nabla \varphi  \right\rangle\, {\rm d}x   \right\}\\
&\leq 2C\eps + \liminf_{k\to \infty}\int_{J_{u_k}\cap A}\left\langle G(u_k^+)-G(u_k^-), \nu_{u_k}\right\rangle^+\,{\rm d}\mathcal{H}^{d-1},
\end{align*}
where the constant $C>0$ depends only on $G$. By the arbitrariness of $\eps>0$, the proof of  (\ref{eq: GSBD LSC intermediate}) is concluded. 
\end{proof}

\BBB While for general  symmetric biconvex functions (see Subsection \ref{sec: biconvex}) the lower semicontinuity in $GSBD^p$ remains an open problem, we point out that the proof strategy devised in Theorem~\ref{thm: GSBD LSC} allows to prove \eqref{eq: GSBD LSC} for symmetric biconvex functions, under the additional assumption that $\lbrace \theta =0 \rbrace = \lbrace 0 \rbrace$. In this case, the natural domain of the energy is $SBD^p$ due to Lemma \ref{lemma: technical}. The statement below is a slight generalization of that in \cite{BCD}, as no $L^\infty$-bound on the sequence has to be assumed.

\begin{theorem}\label{thm: biconv}
Consider a convex and positively $1$-homogeneous function $\theta\colon\R^{d\times d}_{\rm sym}\rightarrow [0,+\infty)$ with $\lbrace \theta =0 \rbrace = \lbrace 0 \rbrace$, and a sequence $(u_k)_k\subset GSBD^p(\Omega)$, $p>1$, converging in measure to $u$, and satisfying the condition
\begin{align}\label{eq: e(u_k) unif bdd}
\sup_{k\in\mathbb{N}} \big( \| e(u_k) \|_{L^p(\Omega)}  + \mathcal{H}^{d-1}(J_{u_k}) \big) <+\infty\,.
\end{align}
Then
\begin{align}
\int_{J_u} \theta([u]\odot\nu_u)\, {\rm d}\mathcal{H}^{d-1} \leq \liminf_{k\to \infty} \int_{J_{u_k}} \theta([u_k]\odot\nu_{u_k})\, {\rm d}\mathcal{H}^{d-1}.
\end{align}
\end{theorem}

\begin{proof} Arguing as in the proof of Proposition \ref{prop: biconvex}, if the right-hand side is finite, the sequence $(u_k)_k$ is bounded in $SBD(\Omega)$, so that $(u_k)_k\in SBD^p(\Omega)$ and  $u\in SBD^p(\Omega)$. In this case, the integration by parts formula \eqref{eq: chaini} still holds for $G(u)=Zu$ with $Z\in \mathbb{R}^{d\times d}_{\rm sym}$, cf.\ \eqref{eq:  bijoint}, thanks to the approximation result in \cite[Theorem 1.1]{Vito}.  With this, the result follows with the same argument as in the proof of Theorem \ref{thm: GSBD LSC}.
\end{proof}

 \EEE

\section*{Acknowledgements} 
This work was supported by the DFG project FR 4083/1-1 and  by the Deutsche Forschungsgemeinschaft (DFG, German Research Foundation) under Germany's Excellence Strategy EXC 2044 -390685587, Mathematics M\"unster: Dynamics--Geometry--Structure. 
The work of Francesco Solombrino is part of the project “Variational methods for stationary and evolution problems with singularities and interfaces” PRIN 2017 financed by the Italian Ministry of Education, University, and Research.

\typeout{References}


\begin{thebibliography}{11}

 





\bibitem{Almgren}
{\sc F.J.~Almgren}. 
\newblock {\em Existence and regularity almost everywhere of solutions to elliptic variational problems with constraints}.
\newblock Mem.\ Am.\ Math.\ Soc.\
\newblock {\bf 165} (1976).







\bibitem{Amb}
{\sc L.~Ambrosio}.
\newblock {\em Existence theory for a new class of variational problems}.
\newblock  Arch.\ Ration.\ Mech.\ Anal.\ 
\newblock {\bf 111} (1990), 291--322.



\bibitem{Ambrosio:90-2}
{\sc L.~Ambrosio}. 
\newblock {\em On the lower semicontinuity of quasi-convex integrals in $SBV(\Omega; \R^k)$}.
\newblock Nonlinear\ Anal.\
\newblock {\bf 23} (1994), 405--425.





\bibitem{AmbrosioBraides}
{\sc L.~Ambrosio, A.~Braides}.
\newblock {\em Functionals defined on partitions of sets of finite perimeter, I: integral
representation and $\Gamma$-convergence}.
\newblock J.\ Math.\ Pures\ Appl.\
\newblock {\bf 69} (1990), 285--305.


\bibitem{AmbrosioBraides2}
{\sc L.~Ambrosio, A.~Braides}.
\newblock {\em Functionals defined on partitions of sets of finite perimeter, II: semicontinuity, relaxation and homogenization}.
\newblock J.\ Math.\ Pures\ Appl.\
\newblock {\bf 69} (1990), 307--333.
 






 \bibitem{ACD}
{\sc L.~Ambrosio, A.~Coscia, G.~Dal Maso}.
\newblock {\em Fine properties of functions with bounded deformation}.
\newblock Arch.\ Ration.\ Mech.\ Anal.\ 
\newblock{\bf 139} (1997), 201--238. 








\bibitem{Ambrosio-Fusco-Pallara:2000} 
{\sc L.~Ambrosio, N.~Fusco, D.~Pallara}.
\newblock {\em Functions of bounded variation and free discontinuity problems}. 
\newblock Oxford University Press, Oxford 2000. 




%
%
%







\bibitem{BCG-part}
 {\sc A.~Braides, S.~Conti, A.~Garroni}.
\newblock  {\em Density of polyhedral partitions}. 
\newblock Calc.\ Var.\ PDE
\newblock  \textbf{56} (2017), art.\ 28.   
 



\bibitem{Braides-Defranceschi:98}
{\sc A.~Braides, A.~Defranceschi}.
\newblock {\em Homogenization of multiple integrals}.
\newblock Oxford University Press, New York 1998.



 
\bibitem{BCD}
{\sc G.~Bellettini, A.~Coscia, G.~Dal Maso}. 
\newblock {\em  Compactness and lower semicontinuity properties in SBD($\Omega$)}.
\newblock Mathematische\ Zeitschrift\ 
\newblock {\bf 228} (1998), 337--351. 
%
%





\bibitem{BFLM}
{\sc G.~Bouchitt\'e, I.~Fonseca, G.~Leoni, L.~Mascarenhas}. 
\newblock {\em A
global method for relaxation in $W^{1,p}$ and in $SBV_p$}. 
\newblock Arch.\ Ration.\ Mech.\ Anal.\
\newblock {\bf 165}  (2002), 187--242. 


\bibitem{BDM}
{\sc G.~Bouchitt\'e, I.~Fonseca, L.~Mascarenhas}. 
\newblock {\em A global method for relaxation}. 
\newblock Arch.\ Ration.\ Mech.\ Anal.\
\newblock {\bf 145} (1998), 51--98. 


 
 
 
\bibitem{Braides:02}
{\sc A.~Braides}.
\newblock {\em $\Gamma$-convergence for Beginners}.
\newblock Oxford University Press, Oxford 2002.

 
 
 
  




 
\bibitem{Caraballo1}
{\sc D.~G.~Caraballo}. 
\newblock {\em  Crystals and polycrystals in $\R^n$: lower semicontinuity and existence}.
\newblock J.\ Geom.\ Anal.\ 
\newblock {\bf 18} (2008), 68--88.
 

 
\bibitem{Caraballo2}
{\sc D.~G.~Caraballo}. 
\newblock {\em The triangle inequalities and lower semi-continuity of surface energy of partitions}.
\newblock Proc.\ Roy.\ Soc.\ Edinburgh Sect.\ A 
\newblock {\bf 139} (2009), 449--457.


\bibitem{Caraballo3}
{\sc D.~G.~Caraballo}. 
\newblock {\em BV-ellipticity and lower semicontinuity of surface energy of Caccioppoli partitions of $\R^n$}.
\newblock J.\ Geom.\ Anal.\
\newblock {\bf 23} (2013), 202--220.







\bibitem{Chambolle-Conti-Francfort:2014}
{\sc A.~Chambolle, S.~Conti, G.~Francfort}.
\newblock {\em Korn-Poincar\'e inequalities for functions
with a small jump set.} 
\newblock Indiana Univ.\ Math.\ J.\
\newblock {\bf 65} (2016), 1373--1399. 




 \bibitem{Chambolle-Giacomini-Ponsiglione:2007}
{\sc A.~Chambolle, A.~Giacomini, M.~Ponsiglione}. 
\newblock {\em Piecewise rigidity}.
\newblock J.\ Funct.\ Anal.\   
\newblock {\bf 244} (2007), 134--153. 





 \bibitem{Iu3}
{\sc A.~Chambolle, S.~Conti, F.~Iurlano}.
\newblock {\em Approximation of functions with small jump sets and existence of strong minimizers of Griffith's energy}. 
\newblock  J.\ Math.\ Pures\ Appl.\
\newblock \BBB {\bf 128} (2019), 119--139. \EEE





 \bibitem{Crismale2}
{\sc A.~Chambolle, V.~Crismale}.
\newblock {\em A density result in $GSBD^p$ with applications to the approximation of brittle fracture energies}. 
\newblock Arch.\ Ration.\ Mech.\ Anal.\ 
\newblock {\bf 232} (2019), 1329--1378.



\bibitem{Crismale}
{\sc A.~Chambolle, V.~Crismale}.
\newblock {\em Compactness and lower semicontinuity in $GSBD$}. 
\newblock     J.\ Eur.\ Math.\ Soc.\ (JEMS), to appear.    Available at: http://cvgmt.sns.it/paper/3767/

 
\bibitem{Crismale3}
{\sc A.~Chambolle, V.~Crismale}.
\newblock {\em Phase-field approximation for a class of cohesive fracture energies with an activation threshold}. 
\newblock     Adv.\ Calc.\ Var., to appear.    Available at:  https://doi.org/10.1515/acv-2019-0018.





\bibitem{Conti-Iurlano:15}
{\sc S.~Conti, M.~Focardi, F.~Iurlano}.
\newblock {\em Which special functions of bounded deformation have bounded variation?} 
\newblock Proc.\ Roy.\ Soc.\ Edinb.\ A.
\newblock {\bf 148} (2018),  33--50.


\bibitem{Conti-Focardi-Iurlano:15}
{\sc S.~Conti, M.~Focardi, F.~Iurlano}.
\newblock {\em Integral representation for functionals defined on $SBD^p$ in dimension two} 
\newblock Arch.\ Ration.\ Mech.\ Anal.\  
\newblock {\bf 223} (2017), 1337--1374.

 	
\BBB


\bibitem{CoFoIu}
{\sc S.~Conti, M.~Focardi, F.~Iurlano}.
\newblock {\em Existence of strong minimizers for the Griffith static  fracture  model  in  dimension  two}. 
\newblock  Ann.\ Inst.\ H.\ Poincar\'e\ Anal.\ Non Lin\'eaire
\newblock {\bf 36} (2019), 455--474.

\EEE 	
 	
 	
 \bibitem{Vito}
{\sc V.~Crismale}.
\newblock {\em On the approximation of $SBD$ functions and some applications}. 
\newblock SIAM.\ J.\ Math.\ Anal.\ 
\newblock {\bf 51} (2019), 5011--5048.





\bibitem{Crismale-Friedrich}
{\sc V.~Crismale, M.~Friedrich}.
\newblock {\em Equilibrium configurations for epitaxially strained films and material voids in three-dimensional linear elasticity.}
\BBB \newblock  Arch. Ration. Mech. Anal.\
\newblock {\bf 237} (2020), 1041--1098. \EEE




 
\bibitem{DalMaso:93}
{\sc G. Dal Maso}.
\newblock {\em An introduction to $\Gamma$-convergence}.
\newblock Birkh{\"a}user, Boston $\cdot$ Basel $\cdot$ Berlin 1993. 






\bibitem{DM}
{\sc G.~Dal Maso}. 
\newblock {\em Generalised functions of bounded deformation}.
\newblock J.\ Eur.\ Math.\ Soc.\ (JEMS)\
\newblock {\bf 15} (2013), 1943--1997. 


\bibitem{DalMOT}
{\sc G.~Dal Maso, G.~Orlando, R.~Toader}. 
\newblock {\em Lower semicontinuity of a class of integral functionals on the space of functions of bounded deformation}.
\newblock Adv.\ Calc.\ Var.\
\newblock {\bf 10} (2017), 183--207. 


\bibitem{DeGiorgi-Ambrosio:1988}
{\sc E.~De Giorgi, L.~Ambrosio}. 
\newblock {\em Un nuovo funzionale del calcolo delle variazioni}. 
\newblock Acc.\ Naz.\ Lincei, Rend.\ Cl.\ Sci.\ Fis.\ Mat.\ Natur.\ 
\newblock {\bf 82} (1988), 199--210. 




\bibitem{Friedrich:15-2}
{\sc M.~Friedrich}.
\newblock {\em A derivation of linearized Griffith energies from nonlinear models}. 
\newblock Arch.\ Ration.\ Mech.\ Anal.\  
\newblock {\bf 225} (2017), 425--467.
 


\bibitem{Friedrich:15-3} 
{\sc M.~Friedrich}.
\newblock {\em A Korn-type inequality in  SBD for functions with small jump sets}. 
\newblock  Math.\ Models Methods Appl.\ Sci.\
 \newblock {\bf 27} (2017), 2461--2484.



\bibitem{Friedrich:15-4}
{\sc M.~Friedrich}.
\newblock {\em A piecewise Korn inequality in SBD and applications to embedding and density results}. 
\newblock  SIAM J.\ Math.\ Anal.\
\newblock {\bf 50} (2018), 3842--3918.



\BBB

\bibitem{FPM}
{\sc M.~Friedrich, M.~Perugini, F.~Solombrino}.
\newblock {\em $\Gamma$-convergence  for free-discontinuity problems in linear elasticity: Homogenization and relaxation}. 
\newblock  Preprint, 2020. Available at:~{\tt https://arxiv.org/abs/2010.05461}.

\EEE

\bibitem{FM}
{\sc M.~Friedrich, F.~Solombrino}.
\newblock {\em Functionals defined on piecewise rigid funtions: integral representation and $\Gamma$-convergence}. 
\BBB \newblock  Arch.\ Ration.\ Mech.\ Anal.\ 
\newblock {\bf 236} (2020), 1325--1387. \EEE

\bibitem{FriedrichSolombrino}
{\sc M.~Friedrich, F.~Solombrino}.
\newblock {\em Quasistatic crack growth in $2d$-linearized elasticity}. 
\newblock  Ann.\ Inst.\ H.\ Poincar\'e\ Anal.\ Non Lin\'eaire
\newblock {\bf 35} (2018), 27--64. 
 




\BBB

\bibitem{GZ1}
{\sc G.~Gargiulo, E.~Zappale}.
\newblock {\em A lower semicontinuity result in $SBD$}. 
\newblock  J.\ Conv.\ Anal.\
\newblock {\bf 15} (2008), 191--200.

\bibitem{GZ2}
{\sc G.~Gargiulo, E.~Zappale}.
\newblock {\em A lower semicontinuity result in $SBD$ for surface integral functionals of fracture mechanics}. 
\newblock Asymptot.\  Anal.\ 
\newblock {\bf 72} (2011),  231-249.


\bibitem{GZ3}
{\sc G.~Gargiulo, E.~Zappale}.
\newblock {\em Some sufficient conditions for lower semicontinuity in $SBD$ and applications to minimum problems of fracture mechanics}. 
\newblock  Math.\  Methods Appl.\  Sci.\
\newblock {\bf 34} (2011),  1541--1552.



\EEE


 

 
\bibitem{Morrey}
{\sc C.~B-~Morrey}.
\newblock {\em Quasiconvexity and lower semicontinuity of multiple integrals}. 
\newblock  Pacific\ J.\ Math.\
\newblock {\bf 2} (1952), 23--53.

 
 
 


\bibitem{matthias}
{\sc M.~Ruf}.
\newblock {\em On the continuity of functionals defined on partitions}. 
\newblock  Adv.\ Calc.\ Var.\
\newblock {\bf 11} (2017), 335--339.





\bibitem{Sim84}
{\sc L.~Simon}, {\em Lectures on geometric measure theory}, vol.~3 of
  Proceedings of the Centre for Mathematical Analysis, Australian National
  University, Australian National University, Centre for Mathematical Analysis,
  Canberra, 1983.



\end{thebibliography}
\end{document}